\documentclass{amsart}
\usepackage{amssymb,amsmath,amsthm}
\usepackage{graphicx}
\usepackage{color}
\usepackage[usenames,dvipsnames,svgnames,table]{xcolor}
\usepackage{comment}
\usepackage{epsfig}
\newtheorem{lemma}{Lemma}[section]
\newtheorem{theorem}[lemma]{Theorem}
\newtheorem{corollary}[lemma]{Corollary}
\newtheorem*{definition}{Definition}

\renewcommand{\epsilon}{\varepsilon}
\renewcommand{\theta}{\vartheta}

\title[Multibrot sets]{Rational Parameter Rays of The Multibrot Sets}

\author[D. Eberlein]{Dominik Eberlein}
\address{Mayor-Huber-Weg 7, 82140 Olching}
\email{dominik.eberlein@kabelmail.de}

\author[S. Mukherjee]{Sabyasachi Mukherjee}
\address{Jacobs University Bremen, Campus Ring 1,  Bremen 28759, Germany}
\email{s.mukherjee@jacobs-university.de, sabyasachi.mukherjee@stonybrook.edu}

\author[D. Schleicher]{Dierk Schleicher}
\address{Jacobs University Bremen, Campus Ring 1,  Bremen 28759, Germany}
\email{d.schleicher@jacobs-university.de}


\date{\today}

\begin{document}
\begin{abstract}
We prove a structure theorem for the multibrot sets, which are the higher degree analogues of the Mandelbrot set, and give a complete picture of the landing behavior of the rational parameter rays and the bifurcation phenomenon. Our proof is inspired by previous works of Schleicher and Milnor on the combinatorics of the Mandelbrot set; in particular, we make essential use of combinatorial tools such as orbit portraits and kneading sequences. However, we avoid the standard global counting arguments in our proof and replace them by local analytic arguments to show that the parabolic and the Misiurewicz parameters are landing points of rational parameter rays.
\end{abstract}

\maketitle

\tableofcontents

\section{Introduction}

The dynamics of quadratic polynomials and their parameter space have been an area of extensive study in the past few decades. The seminal papers of Douady and Hubbard \cite{orsay-notes, DH82} laid the foundation of subsequent works on the topological and combinatorial structures of the Mandelbrot set, which is indeed one of the most complicated objects in the study of dynamical systems.

In this article, we study the multibrot sets $\mathcal{M}_d:= \{ c\in\mathbb{C} :$ The Julia set
of $z^d+c$ is connected$\}$, which are the immediate generalizations of the Mandelbrot set.

The principal goal of this paper is twofold.

(1) We give a new proof of the structure theorem for the Mandelbrot set consisting of a complete description of the landing properties of the rational parameter rays and the bifurcation of hyperbolic components. 

The classical proofs of the structure theorem of the Mandelbrot set can be found in the work of Douady and Hubbard \cite{DH82,orsay-notes}, Schleicher \cite{S1a}, Milnor \cite{M2a}. While Douady-Hubbard's proof involves a careful analysis of the Fatou coordinates and perturbation of parabolic points, more elementary and combinatorial proofs were given by Schleicher and Milnor. The non-trivial part of the structure theorem consists of showing that every parabolic and Misiurewicz parameter is the landing point of the required number of rays, which can be detected by looking at the corresponding dynamical plane. Both the combinatorial proofs are carried out by establishing bounds on the total number of parabolic parameters with given combinatorics and showing that at least (Milnor's proof) or at most (Schleicher's proof) two parameter rays at periodic angles can land at a parabolic parameter. The present proof follows a suggestion from Milnor and avoids the global counting argument. This is replaced by a combination of the combinatorial techniques of \cite{S1a} and Milnor \cite{M2a} (kneading sequences and orbit portraits, respectively) together with a monodromy argument.

(2) We write the proof of the structure theorem in more generality, namely for the multibrot sets. Due to existence of a single critical orbit, the passage from degree two to higher degrees does not add further technical complicacies. However, one needs to look at the combinatorics more carefully and modify some of the proofs in the higher degree unicritical case. 

There has been a growth of interest in the dynamics and parameter spaces of degree $d$ unicritical polynomials (Avila, Kahn and Lyubich \cite{KL,AKLS}, Milnor \cite{M3}, Ch\'{e}ritat \cite{Ch}) in the recent years. A different proof of landing of rational parameter rays of the multibrot sets can be found in \cite{PeRy}. Combinatorial classifications of post-critically finite polynomials in terms of external dynamical rays were given in \cite{BFH,Po}. To our knowledge, there is no written account of the structure theorem for the multibrot sets in the literature and this paper aims at bridging that gap.

It is also worth mentioning that the parameter spaces of unicritical anti-holomorphic polynomials were studied \cite{MNS,Sa,HS,IM} by some of the authors and several combinatorial and topological differences between the connectedness loci of unicritical polynomials and unicritical anti-polynomials (e.g. discontinuity of landing points of dynamical rays, bifurcation along arcs, non-local connectivity of the connectedness loci, non-trivial accumulation of parameter rays etc.) have been discovered. Hence, the present paper also serves as a precise reference for the corresponding properties in the holomorphic setting which facilitates the comparison between these two families.

\begin{figure}[ht!]
\includegraphics[scale=0.24]{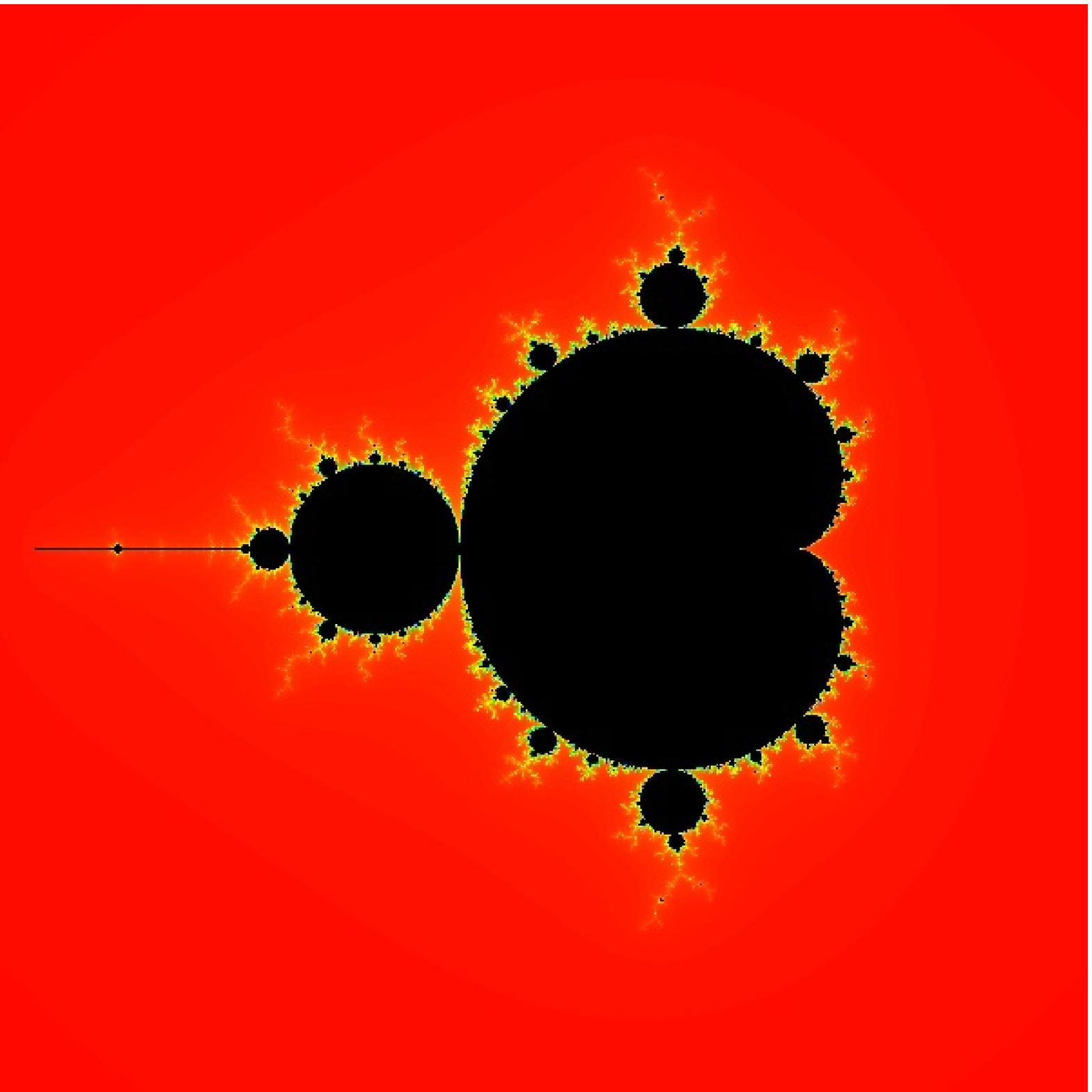} \hspace{1mm} \includegraphics[scale=0.24]{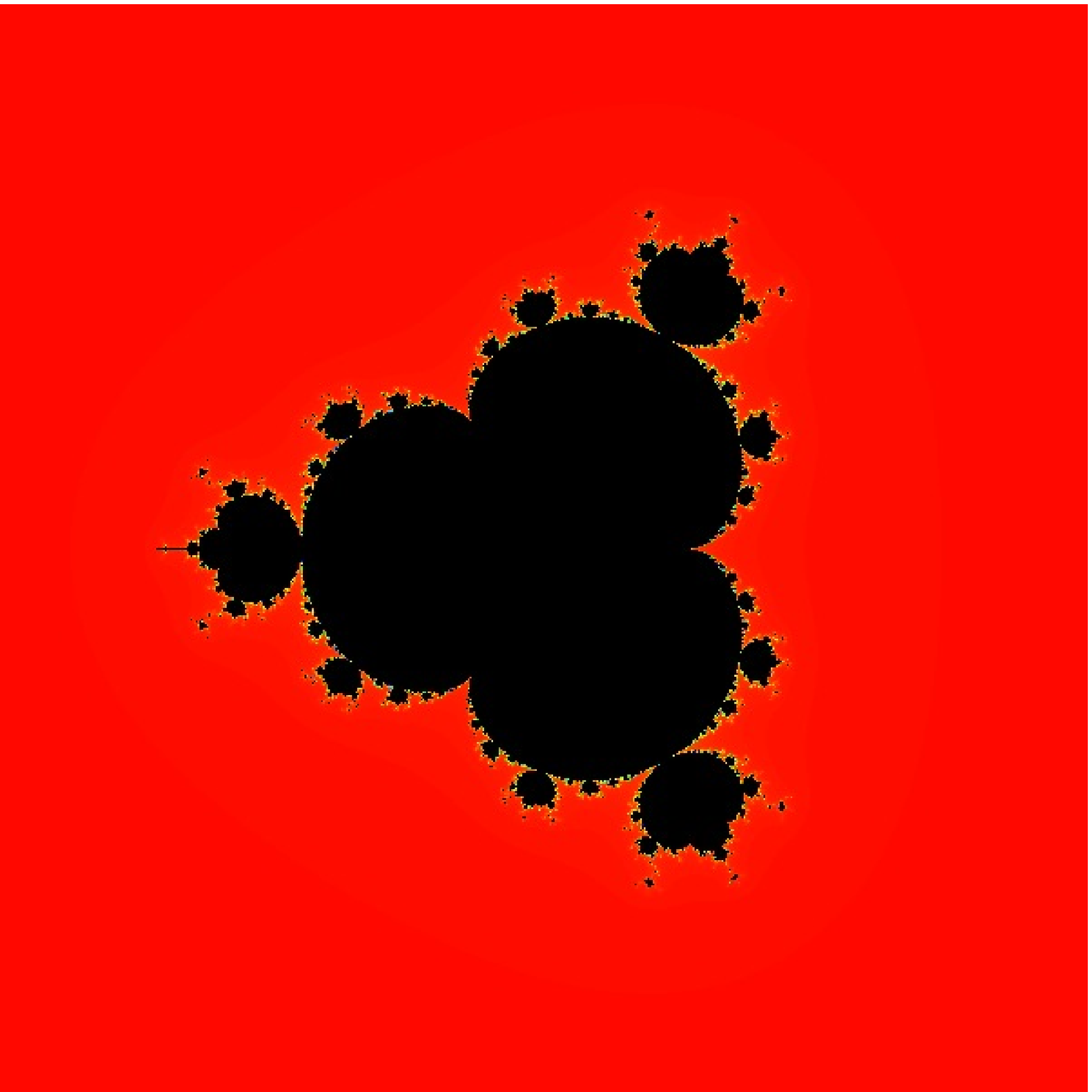}
\caption{Left: The Mandelbrot set, Right: The multibrot set $\mathcal{M}_4$.}
\end{figure}

Let us now review some background material and fix our
notations before stating the main theorem of this paper. For a general introduction into the field, see \cite{M1new}.
In the following, let $d \geq 2$ be an integer, fixed for the whole
paper. Let $f_{c}(z) = z^d + c$ be a unicritical polynomial of degree $d$ (any unicritical polynomial of degree $d$ can be affinely conjugated to a polynomial of the above form). The set of all points which remain bounded under all iterations of $f_c$ is called the Filled-in Julia set $K(f_c).$ The boundary of the Filled-in Julia set is defined to be the Julia set $J(f_c)$ and the complement of the Julia set is defined to be its Fatou set $F(f_c)$.

We measure angles in the fraction of a whole turn, i.e., our angles
are elements of $\mathbb{S}^1 \cong \mathbb{R}/\mathbb{Z}$. We
define for two different angles $\theta_1,\theta_2\in \mathbb{R}/\mathbb{Z}$, the
interval $\left(\theta_1,\theta_2\right)\subset \mathbb{R}/\mathbb{Z}$ as the open connected
component of $\mathbb{R}/\mathbb{Z} \setminus \{ \theta_1,\theta_2 \}$ that consists of the
angles we traverse if we move on $\mathbb{R}/\mathbb{Z}$ in anti-clockwise direction
from $\theta_1$ to $\theta_2$.  Finally, we denote the length of an
interval $I_1 \subset \mathbb{R}/\mathbb{Z}$ by $\ell(I_1)$ such that $\ell(\mathbb{S}_1)=1$.

It is well-known that there is a conformal map $\phi_c$ near $\infty$ such that $\displaystyle \lim_{z\to \infty} \phi_c(z)/z=1$ and $\phi_c\circ f_c(z)= \phi_c(z)^d$. $\phi_c$ extends as a conformal isomorphism to an equipotential containing $0$, when $c\notin \mathcal{M}_d$, and extends as a biholomorphism from $\hat{\mathbb{C}} \setminus K(f_c)$ onto $\hat{\mathbb{C}} \setminus \overline{\mathbb{D}}$ when $c \in \mathcal{M}_d$. The dynamical ray $\mathcal{R}_t^c$ of $f_c$ at an angle $t$ is defined as the pre-image of the radial line at angle $t$ under $\phi_c$.

The dynamical ray $\mathcal{R}_t^c$ at angle $t\in \mathbb{R}/\mathbb{Z}$ maps to the dynamical ray $\mathcal{R}_{dt}^c$ at angle $dt$ under $f_c$. We say that the dynamical ray $\mathcal{R}_t^c$ of $f_c$ lands if $\overline{\mathcal{R}_t^c} \cap K(f_c)$ is a singleton, and this unique point, if exists, is called the landing point of $\mathcal{R}_t^c$.  It is worth mentioning that for a complex polynomial (of degree $d$) with connected Julia set, every dynamical ray at a periodic angle (under multiplication by $d$) lands at a repelling or parabolic periodic point, and conversely, every repelling or parabolic periodic point is the landing point of at least one periodic dynamical ray \cite[\S 18]{M1new}. 

Every $\mathcal{M}_d$ is simply connected and there is a biholomorphic map $\Phi$
from $\mathbb{C} \setminus \mathcal{M}_d$ onto the complement of the closed unit disk $\mathbb{C} \setminus \overline{\mathbb{D}}$ (see \cite{orsay-notes} for a proof in the Mandelbrot set case). The \emph{parameter ray with angle $\theta$} is defined as the set
$\mathcal{R}_\theta :=\{ \Phi^{-1}\left(r^{2\pi i\theta}\right), r>1 \}$
If $\displaystyle \lim_{r \rightarrow 1^+} \Phi^{-1}\left(r^{2\pi i\theta}\right)$ exists, we say
that $\mathcal{R}_\theta$ lands. The parameter rays of the multibrot sets have been profitably used to reveal the combinatorial and topological structure of the multibrot sets. We refer the readers to \cite{lavaurs_systemes_1989} for a combinatorial model of the Mandelbrot set in terms of the external parameter rays. 

For a periodic orbit $\mathcal{O} = \{ z_1, z_2, \cdots, z_p \}$ denote
by $\lambda \left(f,\mathcal{O}\right):= \lambda(f,z):=\frac d{dz}f^{\circ p}(z)$ the
\emph{multiplier of $\mathcal{O}$}.  In the case $f=f_c$ we
write $\lambda(c,\mathcal{O})$ and $\lambda\left(c,z\right)$ instead
of $\lambda\left(f_c,\mathcal{O}\right)$ and $\lambda\left(f_c,z\right)$.

A parameter $c$ in $\mathcal{M}_d$ is called parabolic if the corresponding polynomial has a (necessarily unique) parabolic cycle. A parabolic parameter $c$ is called essential if at least two dynamical rays land at each point of the (unique) parabolic cycle of $f_c$. On the other hand, $c$ is called a non-essential parabolic parameter if exactly one dynamical ray  lands at each point of the (unique) parabolic cycle of $f_c$.

The main theorem of this paper is the following:  

\begin{theorem}[Structure Theorem for Multibrot Sets]
\label{t:struct}
For the Multibrot set~$\mathcal{M}_d$ and the associated parameter rays the
following statements hold:
\begin{enumerate}
\item 
Every parameter ray at a periodic angle lands at a parabolic parameter
of $\mathcal{M}_d$.
\item 
Every essential (resp.\ non-essential) parabolic parameter of $\mathcal{M}_d$ is the 
landing point of exactly two (resp.\ one) parameter ray(s) at periodic angle(s).
\item 
A parameter ray at a periodic angle $\theta$ lands at a parabolic
parameter $c$ if and only if, in the dynamics of $f_c$, the dynamical ray at angle
$\theta$ lands at the parabolic orbit and is one of its characteristic rays.
\item 
Every parameter ray at a pre-periodic angle lands at a post-critically pre-periodic parameter
of $\mathcal{M}_d$.
\item 
Every post-critically pre-periodic parameter is the landing point of at least one
parameter ray at a pre-periodic angle.
\item
A parameter ray at a pre-periodic angle $\theta$ lands at a Misiurewicz parameter
$c$ if and only if, in the dynamics of $c$, the dynamical ray at angle $\theta$ lands
at the critical value.
\item 
Every hyperbolic component of period greater than one of $\mathcal{M}_d$ has exactly one root
and $d-2$ co-roots and the period one hyperbolic component has exactly $d-1$ co-roots and no root.
\end{enumerate}
\end{theorem}

Let us now spend a few words on the organization of the paper. In Section \ref{orbit}, we discuss the basic properties of orbit portraits, which is a combinatorial tool to describe the pattern of all periodic dynamical rays landing at different points of a periodic cycle. This includes a complete description of the orbit portraits and a realization theorem for `formal orbit portraits'. This allows us to define `wakes', which play an important role in the combinatorial structure of the multibrot sets. In Section \ref{s:pardyn}, we explore the duality between parameter rays and dynamical rays. Subsection \ref{Hubbard} contains some basic dynamical and topological properties of Hubbard trees, which contain crucial information about post-critically finite polynomials. This is followed by a preliminary description of centers, roots and co-roots of hyperbolic components and the structure of landing patterns of parameter rays at periodic angles. In Section \ref{s:hyp2} and Section \ref{s:kneading}, we complete the proof of Theorem \ref{t:struct} by giving an exact count of the number of parameter rays at periodic angles landing at parabolic parameters. In Section \ref{s:preper}, we collect the landing properties of parameter rays at pre-periodic angles; the proofs of these results are direct generalizations of the corresponding results for the Mandelbrot set, so we omit the proofs here.

\dedicatory{We thank John Milnor for fruitful discussions and useful advice. The second author gratefully acknowledges the support of Deutsche Forschungsgemeinschaft DFG during this work. The paper is based on the diploma thesis of the first author in spring 1999 at TUM \cite{E2}}.

\section{Orbit Portraits}\label{orbit}
Orbit portraits of quadratic complex polynomials were first introduced by Lisa Goldberg and John Milnor in \cite{Go,GM1,M2a} as a combinatorial tool to describe the pattern of all periodic external rays landing at different points of a periodic cycle. Milnor proved that any collection of finite subsets of $\mathbb{Q}/\mathbb{Z}$ satisfying some conditions indeed occur as the orbit portrait of some quadratic complex polynomial. The usefulness of orbit portraits stems from the fact that these combinatorial objects contain substantial information about the connection between the dynamical and the parameter planes.

\subsection{Definitions and Properties}
In this section, we define orbit portraits for unicritical polynomials of arbitrary degree and prove their basic properties. Finally, we prove an analogous realization theorem for these generalized orbit portraits.

\begin{definition}[Orbit portraits]
Let $\mathcal{O} = \{ z_1 , z_2 ,\cdots,z_p\}$ be a periodic cycle of some unicritical polynomial $f$. If a dynamical ray $\mathcal{R}_t^f$ at a rational angle $t \in \mathbb{Q}/\mathbb{Z}$ lands at some $z_i$; then for all j, the set $\mathcal{A}_j$ of the angles of all the dynamical rays landing at $z_j$ is a non-empty finite subset of $\mathbb{Q}/\mathbb{Z}$. The collection $\{ \mathcal{A}_1 , \mathcal{A}_2, \cdots, \mathcal{A}_p \}$ will be called the \emph{Orbit Portrait} $\mathcal{P(O)}$ of the orbit $\mathcal{O}$ corresponding to the polynomial $f$.
\end{definition}

An orbit portrait $\mathcal{P(O)}$ will be called trivial if only one ray lands at each point of $\mathcal{O}$; i.e. $\vert \mathcal{A}_j \vert = 1,$ for all $j$. Otherwise, the orbit portrait will be called non-trivial. The portrait $\mathcal{P} = \{\{ 0 \}\}$ is also called non-trivial.

\begin{lemma}[Orientation Preservation] For any polynomial $f$, if the external ray $\mathcal{R}_t$ at angle t lands at a point $z \in J(f)$, then the image ray $f \left(\mathcal{R}_t\right) = \mathcal{R}_{dt}$ lands at the point $f(z)$. Furthermore, multiplication by $d$ maps every $\mathcal{A}_j$ bijectively onto $\mathcal{A}_{j+1}$ preserving the cyclic order of the angles around $\mathbb{R}/\mathbb{Z}$. 
\end{lemma}

\begin{proof}
Since the ray $\mathcal{R}_{t}$ lands at $z$, it must not pass through any pre-critical point of $f$, hence the same is true for the image ray $\mathcal{R}_{dt}$. Therefore, the image ray is well-defined all the way to the Julia set and continuity implies that it lands at $f(z)$.
For the second part, observe that $f$ is a local orientation-preserving diffeomorphism from a neighborhood of $z_i$ to a neighborhood of $z_{i+1}$. Hence, it sends the set of rays landing at $z$ bijectively onto the set of rays landing at $f(z)$ preserving their cyclic order.
\end{proof}

\begin{lemma}[Finitely Many Rays] If an external ray at a rational angle lands at some point of a periodic orbit $\mathcal{O}$ corresponding to a polynomial, then only finitely many rays land at each point of $\mathcal{O}$. Moreover, all the rays landing at the periodic orbit have equal period and the ray period can be any multiple of $p$.
\end{lemma}

\emph{Remark.} An angle $t \in \mathbb{R}/\mathbb{Z}$ (resp. a ray $\mathcal{R}_t$) is periodic under multiplication by $d$ (resp. under f) if and only if $t = a/b$ (in the reduced form), for some $ a, b\in \mathbb{N}$ with $\left(b,d\right)=1.$ On the other hand, $t$ (resp. $\mathcal{R}_t$) is strictly pre-periodic if and only if $t = a/b$ (in the reduced form), for some $ a, b\in \mathbb{N}$ with $\left(b,d\right) \neq 1.$

\begin{proof}
This is well-known. See \cite{M1new} for example.
\end{proof}

So far all our discussions hold for general polynomials of degree $d$. Now we want to investigate the consequences of unicriticality on orbit portraits. The next two lemmas are essentially due to Milnor, who proves them for quadratic polynomials in \cite{M2a}.

\begin{lemma}[The Critical Arc]
Let $f$ be a unicritical polynomial of degree $d$ and $\mathcal{O} = \{ z_1 , z_2 ,\cdots,z_p\}$ be an orbit of period $p$. Let $\mathcal{P(O)} = \{ \mathcal{A}_1 , \mathcal{A}_2, \cdots, \mathcal{A}_p \}$ be the corresponding orbit portrait. For each $j \in \{ 1,2,\cdots,p\}$, $\mathcal{A}_j$ is contained in some arc of length $1/d$ in $\mathbb{R}/\mathbb{Z} $. Thus, all but one connected component of $\left(\mathbb{R}/\mathbb{Z}\right) \setminus \mathcal{A}_j $ map bijectively to some connected component of $\left(\mathbb{R}/\mathbb{Z}\right) \setminus \mathcal{A}_{j+1}$ and the remaining complementary arc of $\left(\mathbb{R}/\mathbb{Z}\right) \setminus \mathcal{A}_j $ covers one particular complementary arc of $\mathcal{A}_{j+1}$ $d$-times and all others $(d-1)$-times.
\end{lemma}

\begin{proof}
Let $\theta \in \mathcal{A}_j$. Let $\beta$ be the element of $\mathcal{A}_j$ that lies in $\left[\theta , \theta+1/d\right)$ and is closest to $\left(\theta+1/d\right).$ Similarly, let $\alpha$ be the member of $\mathcal{A}_j$ that lies in $\left(\theta-1/d , \theta \right]$ and is closest to $\left(\theta-1/d\right).$ Note that there is no element of $\mathcal{A}_j$ in $\left(\beta , \beta+1/d\right];$ otherwise the orientation preserving property of multiplication by $d$ would be violated. Similarly, $\left[\alpha-1/d , \alpha \right)$ contains no element of $\mathcal{A}_j$. Also, the arc $\left( \alpha , \beta \right)$ must have length less than $1/d$. We will show that the entire set $\mathcal{A}_j$ is contained in the arc $\left( \alpha , \beta \right)$ of length less than $1/d$.

\begin{figure}[!ht]
\centering
\includegraphics[scale=0.4]{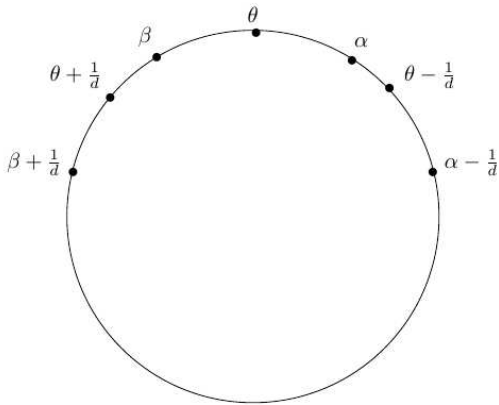}
\caption{No element of $\mathcal{A}_j$ outside $\left(\alpha, \beta\right)$.}
\end{figure}

If there exists some $\gamma \in \mathcal{A}_j$ lying outside $\left( \alpha , \beta \right)$, then $\gamma \in \left( \beta +1/d , \alpha -1/d\right)$. Therefore, there exist at least two complementary arcs of $\mathbb{R}/\mathbb{Z}\setminus \mathcal{A}_j$ of length greater than $1/d$. Both these arcs cover the whole circle and some other arc(s) of $\mathbb{R}/\mathbb{Z}\setminus \mathcal{A}_{j+1}$ under multiplication by $d$. In the dynamical plane of $f$, the two corresponding sectors (of angular width greater than $1/d$) map to the whole plane and some other sector(s) under the dynamics. Therefore, both these sectors contain at least one critical point of $f$. This contradicts the unicriticality of $f$.

This proves that the entire set $\mathcal{A}_j$ is contained in the arc $\left( \alpha , \beta \right)$ of length less than $1/d$.
\end{proof}

\emph{Remark.} Following Milnor, the largest component of $\left(\mathbb{R}/\mathbb{Z}\right) \setminus \mathcal{A}_j $ (of length greater than $(1 - 1/d))$ will be called the \emph{critical arc} of $\mathcal{A}_j$ and the complementary component of $\mathcal{A}_{j+1}$ that is covered $d$-times  by the critical arc of $\mathcal{A}_j$, will be called the \emph{critical value arc} of $\mathcal{A}_{j+1}.$ In the dynamical plane of $f$, the two rays corresponding to the two endpoints of the critical arc of $\mathcal{A}_j$ along with their common landing point bound a sector containing the unique critical point of $f$. This sector is called a \emph{critical sector}. Analogously, the sector bounded by the two rays corresponding to the two endpoints of the critical value arc of $\mathcal{A}_{j+1}$ and their common landing point contains the unique critical value of $f$. This sector is called a \emph{critical value sector}. 

\begin{lemma}[The Characteristic Arc]\label{blah} 
Among all the complementary arcs of the various $\mathcal{A}_j$'s, there is a unique one of minimum length. It is a critical value arc for some $\mathcal{A}_j $ and is strictly contained in all other critical value arcs.
\end{lemma}

\begin{proof}
Among all the complementary arcs of the various $\mathcal{A}_j$'s, there is clearly at least one, say $\left(t^-,t^+\right)$, of minimal length. This arc must be a critical value arc for some $\mathcal{A}_j$: else it would be the diffeomorphic image of some arc of $1/d$ times its length. Let $\left(a,b\right)$ be a critical value arc for some $\mathcal{A}_k$ with $k\neq j$. Both the critical value sectors in the dynamical plane of $f$ contain the unique critical value of $f$. Therefore, $\left(t^-,t^+\right) \bigcap \left(a,b\right)\neq \emptyset$. From the unlinking property of orbit portraits, it follows that $\left(t^-,t^+\right)$ and $\left(a,b\right)$ are strictly nested; i.e. the critical value sector $\left(a,b\right)$ strictly contains $\left(t^-,t^+\right)$.
\end{proof}

This shortest arc $\mathcal{I}_{\mathcal{P}}$ will be called the \emph{Characteristic Arc} of the orbit portrait and the two angles at the ends of this arc will be called the \emph{Characteristic Angles}. The characteristic angles, in some sense, are crucial to the understanding of orbit portraits.

We are now in a position to give a complete description of orbit portraits of unicritical polynomials.

\begin{theorem}\label{complete holomorphic}
$\displaystyle$ Let $f$ be a unicritical polynomial of degree $d$ and $\mathcal{O} = \{ z_1 , z_2 ,\cdots,z_p\}$ be a periodic orbit such that at least one rational dynamical ray lands at some $z_j$. Then the associated orbit portrait (which we assume to be non-trivial) $\mathcal{P(O)} = \{ \mathcal{A}_1 , \mathcal{A}_2, \cdots, \mathcal{A}_p \}$ satisfies the following properties:
\begin{enumerate}
\item Each $\mathcal{A}_j$ is a non-empty finite subset of $\mathbb{Q}/\mathbb{Z}$.

\item The map $\theta\rightarrow d\theta$ sends $\mathcal{A}_j$ bijectively onto $\mathcal{A}_{j+1}$ and preserves the cyclic order of the angles.

\item For each $j, \mathcal{A}_j$ is contained in some arc of length less than $1/d$ in $\mathbb{R}/\mathbb{Z}$.

\item Each $\displaystyle \theta \in \bigcup_{j = 1}^p\mathcal{A}_j$ is periodic under multiplication by $d$ and have a common period $rp$ for some $r\geq 1$.

\item For every $\mathcal{A}_i$, the translated sets $\mathcal{A}_{i,j}:=\mathcal{A}_i+j/d$ ($j=0, 1, 2,\cdots, d-1$) are unlinked from each other and from all other $\mathcal{A}_m$.
\end{enumerate}
\end{theorem}

\begin{proof}
The first four properties follow from the previous lemmas. Property (5) simply states the fact that if two rays
$\mathcal{R}_{\theta}^c$ and $\mathcal{R}_{\theta'}^c$ land together at some periodic point $z$, then the rays $\mathcal{R}_{\theta+j/d}^c$ and $\mathcal{R}_{\theta'+j/d}^c$ land together at a pre-periodic point $z'$ with $f(z')=f(z)$: the Julia set of $f_c$ has a $d$-fold rotation symmetry.
\end{proof}

\begin{definition}[Formal Orbit Portraits]
A finite collection $\mathcal{P(O)} = \{ \mathcal{A}_1 , \mathcal{A}_2, \cdots, \mathcal{A}_p \}$ of subsets of $\mathbb{R}/\mathbb{Z}$ satisfying the six properties of Theorem \ref{complete holomorphic} is called a \emph{formal orbit portrait}.

The condition (3) of Theorem \ref{complete holomorphic} implies that each $\mathcal{A}_j$ has a complementary arc of length greater than $\left( 1 - 1/d\right)$ (which we call the critical arc of  $\mathcal{A}_j$) that, under multiplication by $d$ covers exactly one complementary arc of $\mathcal{A}_{j+1}$ d-times (which we call the critical value arc of $\mathcal{A}_{j+1}$) and the others $\left( d - 1 \right)$-times. We label all the critical value arcs as $\mathcal{I}_1 , \mathcal{I}_2, \cdots, \mathcal{I}_p$. 
\end{definition}

The next lemma is a combinatorial version of Lemma \ref{blah} and this is where condition (5) of the definition of formal orbit portraits comes in.

\begin{lemma}\label{unique smallest arc}
Let $\mathcal{P(O)} = \{ \mathcal{A}_1 , \mathcal{A}_2, \cdots, \mathcal{A}_p \}$ be a formal orbit portrait. Among all the complementary arcs of the various $\mathcal{A}_j$'s, there is a unique one of minimum length. It is a critical value arc for some $\mathcal{A}_j $ and is strictly contained in all other critical value arcs.
\end{lemma}

\begin{proof}
Among all the complementary arcs of the various $\mathcal{A}_j$'s, there is clearly at least one, say $\mathcal{I} =  \left(t^-,t^+\right)$, of minimal length $l$. This arc must be a critical value arc of some $\mathcal{A}_j$:  else it would be the diffeomorphic image of some arc of $1/d$ times its length. Let $\mathcal{I^{\prime}} = \left(a,b\right)$ be the critical arc of $\mathcal{A}_{j-1}$ having length $(d-1)/d+ l/d$ so that its image under multiplication by $d$ covers $\left(t^-,t^+\right)$ $d$-times and the rest of the circle exactly $(d-1)$-times. $\left(t^-,t^+\right)$ has $d$ pre-images $\mathcal{I}/d$, $\left(\mathcal{I}/d+1/d\right)$, $\left(\mathcal{I}/d+2/d\right), \cdots, \left(\mathcal{I}/d+(d-1)/d\right)$; each of them is contained in $\left(a,b\right)$ and has length $l/d$. By our minimality assumption, $\left(t^-,t^+\right)$ contains no element of $\mathcal{P}$ and hence neither do its $d$ pre-images. Label the $d$ connected components of $\displaystyle \mathbb{R}/\mathbb{Z} \setminus \bigcup_{r=0}^{d-1} \left(\mathcal{I}/d+r/d\right)$ as $C_1, C_2,\cdots, C_d$ with $C_1 = \lbrack b,a \rbrack$.

Clearly, $\mathcal{A}_{j-1}$ is contained in $C_1$ and the two end-points of $a$ and $b$ of $C_1$ belong to $\mathcal{A}_{j-1}$. Also, $C_{i+1} = C_1 + i/d$ for $0\leq i \leq d-1$. Therefore, $\mathcal{A}_{j-1}+i/d$ is contained in $C_{i+1}$ with the end-points of $C_{i+1}$ belonging to $\mathcal{A}_{j-1}+i/d$. By condition $\left(5\right)$ of the definition of formal orbit portraits, each $\mathcal{A}_{j-1}+i/d$ (for fixed $j$ and varying $i$) is unlinked from $\mathcal{A}_{k}$, for $k\neq j-1$. This implies that any $\mathcal{A}_k$ $(k \neq j-1)$ is contained in int$ (C_r)$, for a unique $r \in \{ 1, 2, \cdots, d \}$, where $r$ depends on $k$. 

Since any $\mathcal{A}_k$ $(k \neq j-1)$ is contained in int$(C_r)$, for a unique $r \in \{ 1, 2, \cdots, d \}$, all the non-critical arcs of $\mathcal{A}_k$ are contained in the interior of the same $C_r$. Thus all the non-critical value arcs of $\mathcal{A}_{k+1}$ $(k+1 \neq j)$ are contained $\displaystyle \mathbb{R}/\mathbb{Z} \setminus \left[t^-,t^+\right]$. Hence the critical value arc of any $\mathcal{A}_{m}$ $(m \neq j)$ strictly contains $\mathcal{I} =  \left(t^-,t^+\right)$. The uniqueness follows.
\end{proof}

\begin{lemma}\label{folklore}
For a formal orbit portrait $\mathcal{P} = \{ \mathcal{A}_1 , \mathcal{A}_2, \cdots, \mathcal{A}_{p} \}$, multiplication by $d$ either permutes all the angles of $\mathcal{A}_1\cup \mathcal{A}_2\cup \cdots\cup \mathcal{A}_{p}$ or $\vert \mathcal{A}_j \vert \leq 2$ for all $j$ and the first return map of $\mathcal{A}_j$ fixes each angle.
\end{lemma}

\begin{proof}
We assume that the cardinality of each $\mathcal{A}_j$ is at least three and we'll show that multiplication by $d$ permutes all the rays of $\mathcal{P}$. We can also assume that the characteristic arc $\mathcal{I}_{\mathcal{P}}$ is a critical value arc of $\mathcal{A}_1$. Since $\vert \mathcal{A}_1 \vert \geq 3,$ $\mathcal{A}_1$ has at least three complementary components. Let $\mathcal{I}^{+}$ be the arc just to the right of $\mathcal{I}_{\mathcal{P}}$ and $\mathcal{I}^{-}$ be the one just to the left of $\mathcal{I}_{\mathcal{P}}$. Let $\mathcal{I}^{-}$ be longer than $\mathcal{I}^{+}$; i.e. $\textit{l}\left(\mathcal{I}^{-}\right) \geq \textit{l}\left(\mathcal{I}^{+}\right).$ Since $\mathcal{I}^{+}$ is not the critical value arc of $\mathcal{A}_1$, there must exist a critical value arc $\mathcal{I}_c$ which maps diffeomorphically onto $\mathcal{I}^{+}$ under multiplication by $d$; i.e. $\mathcal{I}^{+} = \left( d \right)^m \mathcal{I}_c$, for some $m \geq 1$.

We claim that $\mathcal{I}_c = \mathcal{I}_{\mathcal{P}}$. Otherwise, $\mathcal{I}_c$ will strictly contain the characteristic arc $\mathcal{I}_{\mathcal{P}}$. Since $\mathcal{I}_c$ is strictly smaller than $\mathcal{I}^{+}$, $\mathcal{I}_c$ cannot contain $\mathcal{I}^{+}$. So one end of $\mathcal{I}_c$ must lie in $\mathcal{I}^{+}$; but then it follows from the unlinking property that both ends of $\mathcal{I}_c$ are in $\mathcal{I}^{+}$. Therefore, $\mathcal{I}_c$ contains $\mathcal{I}^{-}$. But this is impossible because $\textit{l}\left(\mathcal{I}^{-}\right) \geq \textit{l}\left(\mathcal{I}^{+}\right) \gneq \textit{l}\left(\mathcal{I}_{c}\right). $

\begin{figure}[!ht]\label{mixed1}
\centering
\includegraphics[scale=0.36]{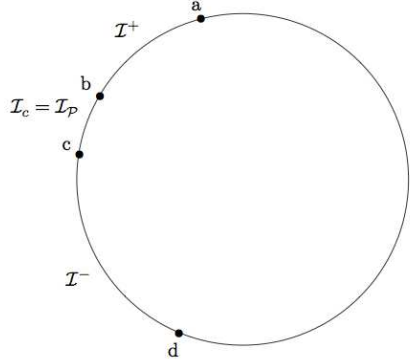}
\caption{The characteristic arc $\mathcal{I}_{\mathcal{P}}$ maps to the shorter adjacent arc $\mathcal{I}^{+}$.}
\end{figure}

Therefore, $\mathcal{I}^{+} = \left( d \right)^m \mathcal{I}_{\mathcal{P}}$. Also let, $\mathcal{I}^{+} = \left( a , b \right)$ and $\mathcal{I}_{\mathcal{P}} = \left( b , c \right)$. Since $\mathcal{I}_{\mathcal{P}}$ maps to $\mathcal{I}^{+}$ by an orientation preserving diffeomorphism, we have: $b = d^{m} c$ and $a = d^{m} b$. Multiplication by $d^m$ is an orientation preserving map and it sends $\mathcal{A}_1$ bijectively onto itself such that the point $b$ is mapped to an adjacent point $a$. It follows that multiplication by $d^m$ acts transitively on $\mathcal{A}_1$. Hence multiplication by $d$ permutes all the rays of $\mathcal{P}$.
\end{proof}

The above dichotomy leads to the following definition:

\begin{definition}[Primitive and Satellite]
If $p$ is the common period of all the angles in $\mathcal{A}_1\cup\ldots\cup A_{p}$, i.e., if each angle is fixed by the first return map, then the portrait is called \emph{primitive}. Otherwise, the orbit portrait is called
\emph{non-primitive} or \emph{satellite}. In the latter case, all angles are permuted transitively under multiplication by $d$.
\end{definition}

We record a few easy corollaries that follow from the proof of the previous lemma.

\begin{corollary}
If $\mathcal{P}$ is a non-trivial formal orbit portrait, the two characteristic angles are on the same
cycle if and only if all angles are on the same cycle. 
\end{corollary}

\begin{corollary}
All angles of a formal orbit portrait are on the orbit of at least one of the characteristic angles.
\end{corollary}

\begin{corollary}[Maximality]\label{maximal}
If $\mathcal{P} = \{ \mathcal{A}_1 , \mathcal{A}_2, \cdots, \mathcal{A}_{p} \}$ is a non-trivial (formal) orbit portrait, then it is maximal in the sense that there doesn't exist any other (formal) orbit portrait $\mathcal{P}^{\prime} = \{ \mathcal{A}_1^{\prime} , \mathcal{A}_2^{\prime}, \cdots, \mathcal{A}_{q}^{\prime} \}$ with $\mathcal{A}_1 \subsetneq \mathcal{A}_1^{\prime}$.
\end{corollary}

\begin{proof}
Suppose there exists non-trivial orbit portraits $\mathcal{P}$ and $\mathcal{P}^{\prime}$ satisfying the above properties. We will consider two cases and obtain contradictions in each of them.
 
Let us first assume that $\mathcal{P}$ is a primitive orbit portrait. Then $\vert \mathcal{A}_1 \vert = 2$ and the two elements of $\mathcal{A}_1$ belong to different cycles under multiplication by $d$. By Lemma \ref{folklore}, $\mathcal{P}^{\prime}$ must necessarily be an orbit portrait of satellite type and all angles of $\displaystyle\bigcup_{i=1}^q \mathcal{A}_i^{\prime}$, in particular, the two elements of $\mathcal{A}_1$ must lie in the same cycle under multiplication by $d$: which is a contradiction.

Now suppose $\mathcal{P}$ is an orbit portrait of satellite type; i.e. multiplication by $d$ permutes all the angles of $\mathcal{P}$ transitively. Let, $\vert \mathcal{A}_1 \vert = v$ (so that the common period of all the angles of $\mathcal{P}$ is $pv$) and $\vert \mathcal{A}_1^{\prime} \vert = v^{\prime}$ (so that the common period of all the angles of $\mathcal{P}^{\prime}$ is $qv^{\prime}$). Since the two portraits have angles in common, it follows that $pv = qv^{\prime}$. The hypothesis $\mathcal{A}_1 \subsetneq \mathcal{A}_1^{\prime}$ implies that $v^{\prime} > v$. Therefore, $q < p$. Clearly, both the sets $\mathcal{A}_1$ and $\mathcal{A}_{1+q} $ are contained in $\mathcal{A}_1^{\prime}$; hence multiplication by $d^p$ would map these two sets onto themselves preserving their cyclic order (multiplication by $d^p$ maps $\mathcal{A}_1^{\prime}$ onto itself). This forces multiplication by $d^p$ to be the identity map on $\mathcal{A}_1$: a contradiction to the transitivity assumption.
\end{proof}

The next lemma gives a necessary and sufficient condition for the characteristic rays of a formal holomorphic orbit portrait to co-land and is a mild generalization of the corresponding result proved in \cite{M2a}.

\begin{lemma}[Outside The Multibrot Sets] Let $\mathcal{P}$ be a formal holomorphic orbit portrait for $d \geq 2$ and $\left(t^-,t^+\right)$ be its characteristic arc. For some  $c$ not in $\mathcal{M}_d,$ the two dynamical rays $\mathcal{R}_{t^-}^c$ and $\mathcal{R}_{t^+}^c$ (where, $f_c = z^d + c$) land at the same point of $J(f_c)$ if the external angle $t(c) \in \left(t^-,t^+\right)$. 
\end{lemma}

\begin{proof}
We retain the terminology of Lemma \ref{unique smallest arc}. If $c \notin \mathcal{M}_d$, all the periodic points of $f_c = z^d + c$ are repelling and the Julia set is a cantor set.

Let the external angle of $c$ in the parameter plane be $t(c)$. Label the connected components of $\displaystyle \mathbb{R}/\mathbb{Z}\setminus \{ t(c)/d , t(c)/d + 1/d, \cdots, t(c)/d + (d-1)/d\}$ counter-clockwise as $L_0, L_1, \cdots, L_{d-1}$ such that the component containing the angle $`0$' gets label $L_0$ . The $t(c)$-itinerary of an angle $\theta \in \mathbb{R}/\mathbb{Z}$ is defined as a sequence $\left(a_n\right)_{n \geq 0}$ in $\{ 0 , 1, \cdots, d-1 \}^{\mathbb{N}}$ such that $a_n = i$ if $d^n\theta \in L_i$. All but a countably many $\theta$'s (the ones which are not the iterated pre-images of $t(c)$ under multiplication by $d$) have a well-defined $t(c)$-itinerary. 


Similarly, in the dynamical plane of $f_c$, the $d$ external rays $\displaystyle \mathcal{R}_{t(c)/d}^{f_c}$, $\mathcal{R}_{t(c)/d + 1/d}^{f_c},$  $\cdots, \mathcal{R}_{t(c)/d + (d-1)/d}^{f_c}$ land at the critical point $0$ and cut the dynamical plane into $d$ sectors. Label these sectors counter-clockwise as $L^{\prime}_0, L^{\prime}_1, \cdots, L^{\prime}_{d-1}$ such that the component containing the external ray $\mathcal{R}_{0}^{f_c}$ at angle $`0$' gets label $L^{\prime}_0$. Any point $z \in J(f_c)$ has an associated symbol sequence $\left(a_n\right)_{n \geq 0}$ in $\{ 0 , 1, \cdots, d-1 \}^{\mathbb{N}}$ such that $a_n = i$ if $f_c^{\circ n}(z) \in L^{\prime}_i$. Clearly, a dynamical ray $\mathcal{R}_{\theta}^{f_c}$ at angle $\theta$ lands at $z$ if and only if the $t(c)$-itinerary of $\theta$ coincides with the symbol sequence of $z$ defined above.

If $t(c) \in \mathcal{I} = \left(t^-,t^+\right)$, the $d$ angles $\{ t(c)/d , t(c)/d + 1/d, \cdots, t(c)/d + (d-1)/d\}$ lie in the $d$ intervals $\mathcal{I}/d$, $\left(\mathcal{I}/d+1/d\right), \cdots, \left(\mathcal{I}/d+(d-1)/d\right)$ respectively and no element of $\mathcal{P}$ belongs to $\displaystyle \bigcup_{j=0}^{d-1} \left(\mathcal{I}/d+j/d\right)$. First note that the rays $\mathcal{R}_{t^-}^c$ and $\mathcal{R}_{t^+}^c$ indeed land at $J(f_c)$ as $t(c) \notin$ the finite sets $\{  t^{\pm}, dt^{\pm}, d^2t^{\pm},\cdots \}$. Each $\mathcal{A}_j$ is contained a unique $C_r.$ Therefore, for each $n\geq 0$, the angles $d^n t^-$ and $d^n t^+$ belong to the same $L_i$. So $t^-$ and $t^+$ have the same $t(c)$-itinerary; which implies that the two characteristic rays $\mathcal{R}_{t^-}^c$ and $\mathcal{R}_{t^+}^c$ land at the same point of $J(f_c).$
\end{proof}

It is now easy to prove the realization theorem for formal orbit portraits.

\begin{theorem}[Realizing Orbit Portraits]\label{realization}
Let $\mathcal{P(O)} = \{ \mathcal{A}_1 , \mathcal{A}_2, \cdots, \mathcal{A}_p \}$ be a formal orbit portrait for some $d \geq 2$. Then there exists some $c \in \mathbb{C}\setminus\mathcal{M}_d$, such that $f(z)= z^d + c$ has a repelling periodic orbit with associated orbit portrait $\mathcal{P(O)}$.
\end{theorem}

\begin{proof}
Let $\mathcal{P} = \{ \mathcal{A}_1 , \mathcal{A}_2, \cdots, \mathcal{A}_{p} \}$ be a formal orbit portrait with characteristic arc $\mathcal{I}_{\mathcal{P}}=\left(t^-,t^+\right)$ such that $\{ t^-,t^+ \} \subset \mathcal{A}_1$.  Choose $c$ outside the Multibrot set $\mathcal{M}_d$ with $t(c) \in \left(t^-,t^+\right)$. Then, the two dynamical rays $\mathcal{R}_{t^-}^c$ and $\mathcal{R}_{t^+}^c$ land at the same point $z \in J\left(f_c\right)$. Let $\mathcal{P}^{\prime} = \{ \mathcal{A}_1^{\prime} , \mathcal{A}_2^{\prime}, \cdots, \mathcal{A}_{p^{\prime}}^{\prime} \}$ be the orbit portrait associated with $\mathcal{O}(z)$ such that $\mathcal{A}_1^{\prime}$ is the set of angles of the rays landing at $z$.

First let's assume that $\vert \mathcal{A}_j \vert = 2$ and the first return map (multiplication by $d^p$) of $\mathcal{A}_j$ fixes each ray. In this case, $\{ t^-,t^+ \} = \mathcal{A}_1$ and $\mathcal{A}_1 \subset \mathcal{A}_1^{\prime}$. By Lemma \ref{maximal}, we have $\mathcal{P}^{\prime} = \mathcal{P}$.

On the other hand, if multiplication by $d$ acts transitively on $\mathcal{P}$, there exists $l \in \mathbb{N}$ such that $d^{lp} t^- = t^+$. It follows that $f_c^{\circ lp}(z) = z$. Since $t^-$ and $t^+$ are adjacent angles in $\mathcal{A}_1$, it follows that multiplication by $d^{lp}$ acts transitively on $\mathcal{A}_1$ and all the rays in $\mathcal{A}_1$ land at $z$. Evidently, $\mathcal{A}_1 \subset \mathcal{A}_1^{\prime}$ and once again Lemma \ref{maximal} implies that $\mathcal{P}^{\prime} = \mathcal{P}$.
\end{proof}

\begin{corollary}[Characteristic Angles Determine Formal Orbit Portraits]
\label{Lem:Characteristic angles}
Let  $\mathcal{P}$ be a non-trivial formal orbit portrait with characteristic angles
$t^-$ and $t^+$. Then a formal orbit portrait $\mathcal{P}'=\{\mathcal{A}'_1,\cdots,\mathcal{A}'_{p}\}$
equals $\mathcal{P}$ if and only if some $\mathcal{A}'_i$ contains $t^-$ and $t^+$.
\end{corollary}

\begin{proof}
Follows from the proof of the previous theorem.
\end{proof}

\emph{Remark.} Conversely, it is easy to show that for some $c$ not in $\mathcal{M}_d$, the unicritical polynomial $f_c$ can admit the orbit portrait $\mathcal{P} = \{ \mathcal{A}_1 , \mathcal{A}_2, \cdots, \mathcal{A}_{p} \}$ only if $t(c) \in \left( t^- , t^+ \right)$, where $\left( t^- , t^+ \right)$ is the characteristic arc of $\mathcal{P}$. Indeed, the characteristic arc must be a critical value arc for some $\mathcal{A}_j$ and in the dynamical plane, the corresponding critical value sector bounded by the two rays $\mathcal{R}_{t^-}^c$ and $\mathcal{R}_{t^+}^c$ together with their common landing point contains the critical value $c$. Therefore the external angle $t(c)$ of $c$ will lie in the interval $\left( t^- , t^+ \right)$.

\begin{lemma}[Characteristic Point]
\label{CharPoint}
Every periodic orbit with a non-trivial orbit portrait has a unique point $z$, called the
\emph{characteristic point of the orbit}, with the following property: two
dynamical rays landing at $z$ separate the critical value from the critical point,
from all points of the orbit other than $z$, and from all other dynamical rays
landing at the orbit of $z$. The external angles of these two rays are exactly the
characteristic angles of the orbit portrait associated with the orbit of $z$.
\end{lemma}

\begin{proof}
The characteristic rays land at a common point $z$ of the orbit and divide $\mathbb{C}$
into two open complementary components. By definition, one of the domains, say $U_1$
contains exactly the external angles from the characteristic arc $\mathcal{I}_{\mathcal{P}}$; let the
other component be $U_0$. Then clearly the closure $\overline{U_0}$ must contain all
dynamical rays with angles from the portrait, and hence the entire orbit of $z$. The
critical value must be contained in $U_1$ (or $U_1$ would have a pre-image bounded
by rays from the portrait, yielding a complementary arc of the portrait which
was shorter than the characteristic arc). Also, the critical point must be contained
in $U_0$. Indeed, if $0$ belonged to $U_1$, then $U_1$ would be a critical sector and  $\mathcal{I}_{\mathcal{P}}$ would be a critical arc of $\mathcal{P}$. But a critical arc has length greater than $\left( 1 - 1/d \right)$ and $\mathcal{I}_{\mathcal{P}}$ being the characteristic arc, has the smallest length amongst all complementary arcs of $\mathcal{P}$. Clearly, the smallest complementary arc of  $\mathcal{P}$ cannot have length greater than $\left( 1 - 1/d \right)$. This proves the lemma.
\end{proof}

\subsection{Stability of Orbit Portraits}

\begin{lemma}[Landing of Dynamical Rays]
\label{Lem:Landing of Dyn Rays}
For every map $f_c$ and every periodic angle $t$ of some period $n$, the
dynamical ray $\mathcal{R}_{t}^c$ lands at a repelling periodic point of period dividing
$n$, except in the following circumstances:
\begin{itemize}
\item
$c \in \mathcal{M}_d$ and $\mathcal{R}_{t}^c$ lands at a parabolic orbit;
\item
$c \notin  \mathcal{M}_d$ is on the parameter ray at some angle $d^kt$, for $k \in \mathbb{N}$.
\end{itemize}
Conversely, every repelling or parabolic periodic point is the landing point of at
least one periodic dynamical ray for every $c \in \mathcal{M}_d$.
\end{lemma}
\begin{proof}
This is well known. For the case $c \in \mathcal{M}_d$, see \cite[Theorem 18.10]{M1new}; if
$c \notin \mathcal{M}_d$, see \cite[Appendix A]{GM1}. For the converse, see
\cite[Theorem 18.11]{M1new}.
\end{proof}

We should note that for every periodic point $z_0$ of $f_{c_0}$
with multiplier $\lambda(c_0,z_0) \neq 1$, the orbit of $z_0$ remains
locally stable under perturbation of the parameter by the Implicit
Function Theorem. With some restrictions this is even true for the
associated portrait.

\begin{lemma}[Stability of Portraits]\label{l:preserving-portrait1}
Let $c_0$ be a parameter such that $f_{c_0} = z^d + c$ has a repelling periodic point $z_0$ so that the rays at angles $A := \{ t_1,\cdots,t_v \}$ $(v \geq 2)$ land at $z_0$. Then there
exist a neighborhood $U$ of $c_0$ and a unique holomorphic function $z \colon U \to \mathbb{C}$
with $z(c_0)=z_0$ so that for every $c\in U$ the point $z(c)$ is a repelling
periodic point for $f_{c}$ and all rays with angles in $A$ land at $z(c)$.

Let $n$ be the common period of the rays in $A$. Let $U$ be an open, path connected neighborhood of $c_0$ which is disjoint from all parabolic parameters of ray period $n$ and from all parameter rays at angles in $\tilde A:=\bigcup_{j\geq 0}d^{j}A$., then for all parameters $c\in U$, exactly the rays at angles in $A$ (and no more) co-land. The same is true at any parabolic parameter of ray period $n$ on the boundary of $U$.
\end{lemma}

\begin{proof}
The point $z_0$ can be continued analytically as a repelling periodic point of
$f_{c_0}$ in a neighborhood of $c_0$. By \cite[Lemma B.1]{GM1}, this orbit will
keep its periodic dynamical rays in a neighborhood of $c_0$; it might possibly gain
extra rays.

Since $U$ doesn't intersect any dynamical ray of period $n$, for all $c \in U$, the dynamical rays in $A$ indeed land. Let $U^{\prime}$ be the subset of $U$ where the rays in $A$ continue to land together. For any $c \in U^{\prime}$, the common landing point of the rays in $A$ must be repelling as $U$ doesn't contain any parabolic point of ray period $n$. By local stability of co-landing rays at repelling periodic points, we conclude that $U^{\prime}$ is open in $U$.

Let $c$ be a limit point of $U^{\prime}$ in $U$. The external rays at angles in $A$ do land in the dynamical plane of $f_c$; we claim that they all co-land. Otherwise, at least two rays at angles $\theta_1$ and $\theta_2$ (say) in $A$ land at two different repelling periodic points in the dynamical plane of $f_c$ and by implicit function theorem, these two rays would continue to land at different points for all nearly parameters. This contradicts the fact that $c$ lies on the boundary of $U^{\prime}$. Hence, $U^{\prime}$ is closed in $U$. Therefore, $U^{\prime} = U$; i.e. the rays in $A$ co-land throughout $U$.

Finally, it follows from the maximality property (Lemma \ref{maximal}) of non-trivial orbit portraits that no dynamical ray at angle $\theta \notin A$ can co-land along with the rays at angles in $A$.
\end{proof}

The local stability of external rays landing at repelling periodic points is true for arbitrary polynomials (not necessarily unicritical) as well. However, there is a danger that such an orbit gains periodic rays: if for a parameter
$c_0$ some periodic ray lands at a parabolic orbit, then under small perturbations all continuations
of the parabolic orbit may lose the periodic ray, and this ray can land at a different repelling orbit
for all sufficiently small (non-zero) perturbations. This happens for general cubic or biquadratic polynomials (compare \cite[\S 6]{MNS}).

\subsection{Wakes}
We begin a basic lemma which states that the set of parabolic parameters are isolated.
 
\begin{lemma}[Parabolic Parameters Are Countable]
\label{l:parcount}
For $n \in \mathbb{N}$ the number of parameters with parabolic orbits of ray period $n$ is
finite.
\end{lemma}

\begin{proof}
Let $Q(c,z):=f_c^{\circ n}(z) - z$, considering it as a polynomial in $z$ whose coefficients are polynomials in $c$. If $f_c$ has a parabolic orbit of ray period $n$, then we have: $f_c^{\circ n}(z)=z$ and $\frac d{dz}f_c^{\circ n}(z)=1$.

In other terms, this reads: $Q(c,z)=0$ and $\frac d{dz}Q(c,z)=0$. Therefore, for such a parameter $c$, $Q(c,z)$ has a multiple root in $z$ forcing its discriminant to vanish. The discriminant of $Q(c,z)$ (viewed as a polynomial in $z$) is simply a polynomial in $c$. Therefore, there are only finitely many such values of $c$, which finishes the proof.
\end{proof}

The previous lemma is false if ``ray period $n$'' is replaced by ``orbit period
$n$'': the boundary of every hyperbolic component of period $n$ contains a dense set
of parabolic parameters with orbit period $n$.

As a consequence of the previous lemma and the stability of orbit portraits near repelling periodic points, we deduce the fundamental fact that every parameter ray at a periodic angle lands.

\begin{lemma}[Periodic Parameter Rays Land]
\label{per-rays-land}
Let $\theta$ be an angle of exact period $n$ under multiplication by $d$. Then the parameter
ray $\mathcal{R}_{\theta}$ lands at a parabolic parameter $c_0$ with parabolic
orbit of exact ray period $n$ such that the dynamical ray $\mathcal{R}_{\theta}^{c_0}$
lands at a point of the parabolic orbit.
\end{lemma}

\begin{proof}
We follow the method of Goldberg and Milnor (\cite[Theorem C.7]{GM1}). 
Let $c \in \mathcal{M}_d$ be a limit point of the parameter ray $\mathcal{R}_{\theta}$. If the dynamical
ray $\mathcal{R}_{\theta}^c$ landed at a repelling periodic point, then it would continue to do
so in a neighborhood $U$ of $c$ by Lemma  \cite[Lemma B.1]{GM1}. But for any parameter on the
parameter ray $\mathcal{R}_{\theta}$, this is impossible since for such a parameter, the dynamical ray at angle $\theta$ bounces off a pre-critical point and fails to land. Therefore, by Lemma \ref{Lem:Landing of Dyn Rays}, the dynamical ray $\mathcal{R}_{\theta}^c$ must land at a parabolic periodic point and $c$ is one of the finitely many parabolic
parameters with a parabolic orbit of ray period $n$ (Lemma \ref{l:parcount}).
Since the set of limit points of any ray is connected, the claim follows.
\end{proof}

\begin{figure}[!ht]\label{wake}
\begin{minipage}{0.48\linewidth}
\begin{center}
\includegraphics[scale=0.24]{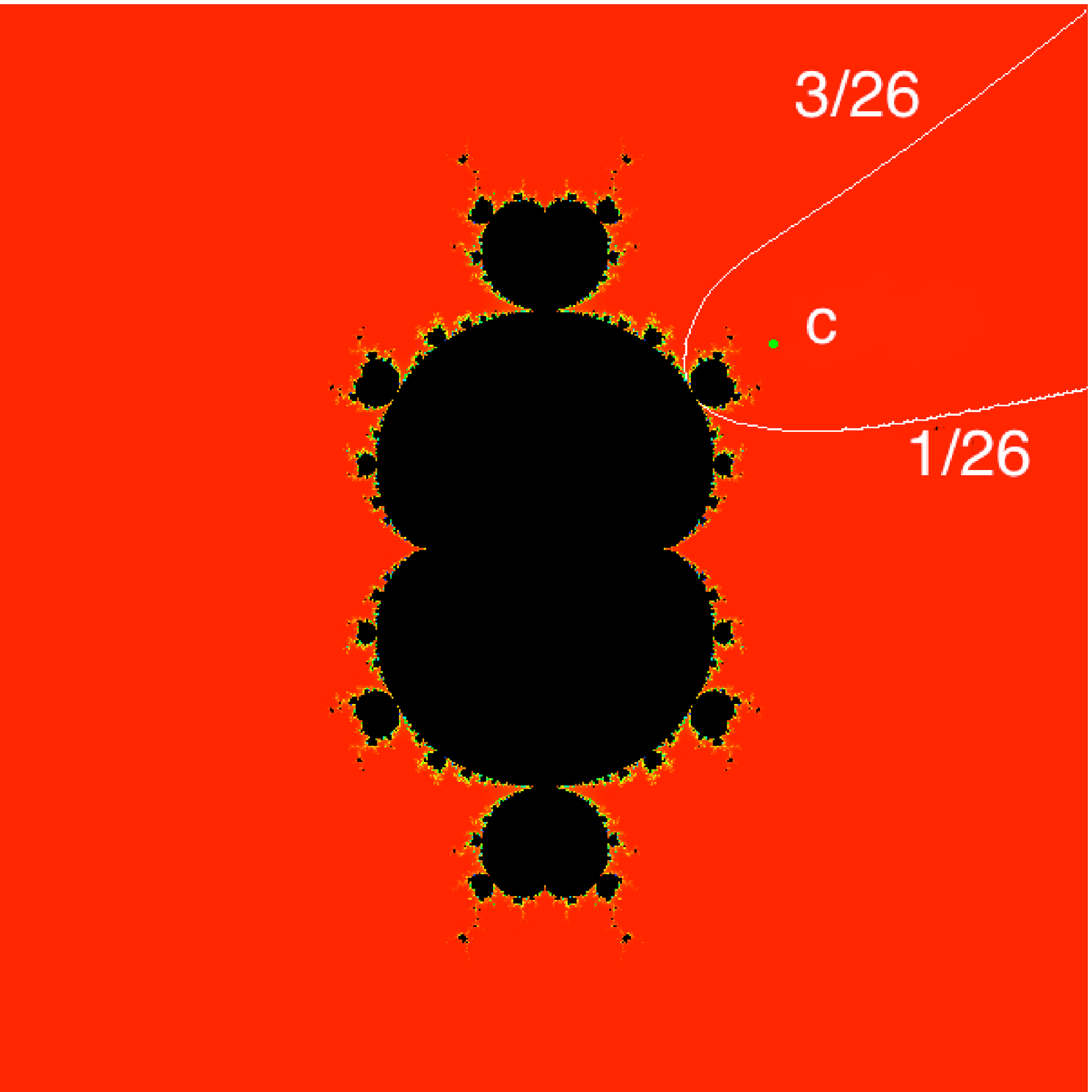}
\end{center}
\end{minipage}
\begin{minipage}{0.48\linewidth}
\begin{center}
\includegraphics[scale=0.24]{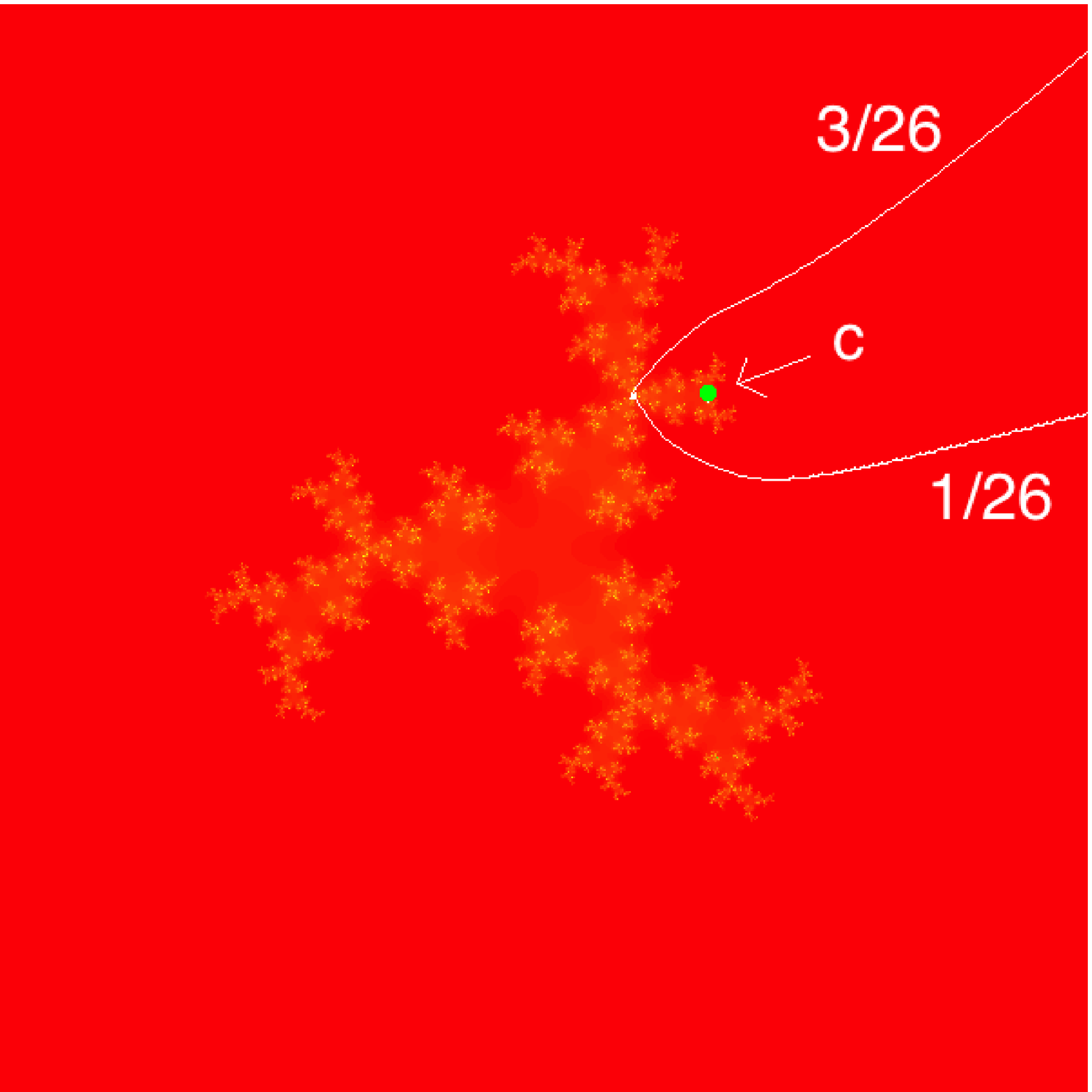}
\end{center}
\end{minipage}
\caption{Left : The $\mathcal{P}$-wake of~$\mathcal{M}_3$ for the portrait~$\mathcal{P}$ with the
characteristic interval~$\left(1/26,3/26\right)$. Right : The Julia set of~$z \mapsto z^3+c$ for a parameter~$c$ in the $\mathcal{P}$-wake.}
\end{figure}

\begin{theorem}[Parameter Rays Landing at a Common Point]
\label{t:p-wake}
Let $\mathcal{P}$ be a non-trivial portrait with characteristic
angles $t^{-}$ and $t^{+}$. Then the parameter
rays $\mathcal{R}_{t^{-}}$ and $\mathcal{R}_{t^{+}}$ land at a common
parabolic parameter.
\end{theorem}

\begin{proof}
For $\mathcal{P} = \{ A_1,\cdots,A_{p}\}$, let $n$ be the common period of
the angles in $A_1\cup\ldots\cup A_{p}$ and let $F_n$ be the
set of all parabolic parameters of ray period $n$.
Consider the connected components $U_i$ of $\displaystyle \mathbb{C} \setminus (\bigcup_{t\in A_1\cup\cdots\cup A_{p}} \mathcal{R}_t \cup F_n)$. Since $\bigcup A_i$ and $F_n$ are finite (Lemma \ref{l:parcount}) there are
only finitely many components and by Proposition \ref{per-rays-land}
they are open. By Lemma \ref{l:preserving-portrait1}, throughout every component $U_i$ 
the same rays with angles in $A_1\cup\ldots\cup A_{p}$ land at common points.
 
Let $U_1$ be the component which contains all parameters $c$ outside $\mathcal{M}_d$ with external angle $t(c)\in (t^{-},t^{+})$ (there is such a component as $(t^{-}, t^{+})$ does not contain any other angle of $\mathcal{P}$).  $U_1$ must have the two parameter rays $\mathcal{R}_{t^+}$ and $\mathcal{R}_{t^-}$ on its boundary.  By Theorem \ref{realization} and Lemma \ref{l:preserving-portrait1}, each $c \in U_1\setminus \mathcal{M}_d$ has a repelling periodic orbit admitting the portrait $\mathcal{P}$. If the two parameter rays at angles $t^+$ and $t^-$ do not co-land, then $U_1$ would contain parameters $c$ outside $\mathcal{M}_d$ with $t(c)\notin (t^{-},t^{+})$. It follows form the remark at the end of Theorem \ref{realization} that such a parameter can never admit the orbit portrait $\mathcal{P}$, a contradiction. Hence, the parameter rays $\mathcal{R}_{t^+}$ and $\mathcal{R}_{t^-}$ at the characteristic angles must land at the same point of $\mathcal{M}_d$.
\end{proof}

\begin{definition}[Wake]
\label{d:wake}
For an essential portrait $\mathcal{P}$ with characteristic angles $t^-$ and
$t^+$, the \emph{wake} of the portrait (or simply the $\mathcal{P}$-wake of $\mathcal{M}_d$) is
defined to be the component of $\mathbb{C}\setminus (\mathcal{R}_{t^{-}}\cup
\mathcal{R}_{t^{+}}\cup\{ c_0\})$ not containing $0$, where $c_0$ is the common
landing point of $\mathcal{R}_{t^{-}}$ and $\mathcal{R}_{t^{-}}$. The wake associated with the orbit portrait $\mathcal{P}$ is denoted by $\mathcal{W}_{\mathcal{P}}$.
\end{definition}

The wake is well-defined because the parameter $0$ can never be on the boundary
of the partition.

\begin{lemma}[Portrait Realized Only in Wake]
\label{c:wake}
For every non-trivial portrait $\mathcal{P}$, the associated wake $\mathcal{W}_{\mathcal{P}}$ is exactly the locus of
parameters $c$ for which the map $f_c$ has a repelling orbit with portrait $\mathcal{P}$.
\end{lemma}

\begin{proof}
The proof is similar to \cite[Theorem 3.1]{M2a}, so we only give a sketch omitting the details. 

We use the notations of the previous theorem. If $c$ belongs to $\mathcal{W}_{\mathcal{P}}\setminus \mathcal{M}_d$, then $t(c)\in (t^-,t^+)$. By the proof of Theorem \ref{realization}, such an $f_c$ has a repelling periodic orbit with associated orbit portrait $\mathcal{P}$. Since $\mathcal{W}_{\mathcal{P}}\setminus F_n$ is disjoint from all parabolic parameters of ray period $n$ and from all parameter rays at angles in $A_1\cup\cdots\cup A_{p}$, it follows from Lemma \ref{l:preserving-portrait1} that every $c$ in $\mathcal{W}_{\mathcal{P}}\setminus F_n$ has a repelling cycle admitting the portrait $\mathcal{P}$. By the same reasoning (using the fact that for some $c$ not in $\mathcal{M}_d$, $f_c$ can admit the orbit portrait $\mathcal{P}$ only if $t(c) \in (t^- , t^+)$), no parameter in $\mathbb{C}\setminus (\mathcal{W}_{\mathcal{P}}\cup F_n)$ has a repelling cycle admitting the portrait $\mathcal{P}$. If for some $c_0$ in $F_n\setminus \mathcal{W}_{\mathcal{P}}$, $f_{c_0}$ has a repelling cycle with associated portrait $\mathcal{P}$, then by Lemma \ref{l:preserving-portrait1}, each nearby map $f_c$ would have a repelling cycle with associated portrait $\mathcal{P}$. But this contradicts the fact that $F_n$ is finite and no parameter in $\mathbb{C}\setminus (\mathcal{W}_{\mathcal{P}}\cup F_n)$ has a repelling cycle admitting the portrait $\mathcal{P}$. The treatment of the parabolic parameters $F_n\cap \mathcal{W}_{\mathcal{P}}$ is slightly more subtle: one can show that the landing points of the dynamical rays $\mathcal{R}_{t^-}^c$ and $\mathcal{R}_{t^+}^c$ depend holomorphically throughout $\mathcal{W}_{\mathcal{P}}$; since these two rays land at a common repelling periodic point for each $c$ in $\mathcal{W}_{\mathcal{P}}\setminus F_n$, it follows from holomorphicity that $\mathcal{R}_{t^-}^c$ and $\mathcal{R}_{t^+}^c$ land at a common repelling periodic point for any $c$ in $F_n\cap \mathcal{W}_{\mathcal{P}}$. Now it follows from the proof of Theorem \ref{realization} that for every $c$ in $F_n\cap \mathcal{W}_{\mathcal{P}}$, $f_c$ has a repelling cycle admitting the portrait $\mathcal{P}$.
\end{proof}

\section{Parameter and Dynamical Rays at the Same Angle}\label{s:pardyn}

\begin{definition}[Characteristic Point of a Parabolic Orbit]
Let $\mathcal{O}$ be a parabolic orbit for some $f_c$. The unique point on this orbit which lies on the boundary of the bounded Fatou component containing the critical value is defined as the \emph{Characteristic point of} $\mathcal{O}$. 

It is easy to see that if $\mathcal{O}$ has a non-trivial orbit portrait, this definition coincides with the one in \ref{CharPoint}.
\end{definition}

Consider the \emph{parameter} ray $\mathcal{R}_\theta$ at a periodic angle $\theta$ with
landing point $c$. We show in this section that the corresponding 
\emph{dynamical} ray $\mathcal{R}_\theta^c$ lands at the characteristic point of
the parabolic orbit (Theorem~\ref{t:perchar}). This proves the
Structure Theorem \ref{t:struct} in the primitive case (Corollary \ref{c:nonesspar})
and ``half'' of the Structure Theorem in the non-primitive case
(Lemma \ref{At least two rays}).

The main tool is the concept of orbit separation
(Subsection \ref{ss:osl}), which in turn is based on Hubbard trees as introduced in
the Orsay Notes \cite{orsay-notes}

\subsection{Hubbard Trees}\label{Hubbard}
Recall the following fact for a parameter $c$ with super-attracting
orbit: for every Fatou component $U$ of $K_c$ there is a unique periodic
or pre-periodic point $z_U$ of the super-attracting orbit and a Riemann
map $\phi_U \colon U\to \mathbb{D}$ with $\phi_U(z_U) = 0$ that extends to a homeomorphism
from $\overline U$ onto $\overline{\mathbb{D}}$. Such a map $\phi_U$ is unique except for
rotation around $0$. The point $z_U$ is called the \emph{center of $U$} and for
any $\theta\in \mathbb{R}/\mathbb{Z}$ the pre-image $\{ \phi_U^{-1}(r^{2\pi
i\theta}) : r \in \left[0,1\right)\}$ is an \emph{internal ray of $U$} with
well-defined landing point. Moreover, since hyperbolic filled-in Julia
sets are connected and locally connected, hence arcwise
connected \cite[Lemmas~17.17 and 17.18]{M1new}, every two different points
$z,z^{\prime} \in K_c$ are connected by an arc in $K_c$ with endpoints $z$ and $z^{\prime}$ (an
arc is an injective path). 

\begin{definition}[Regular Arc]
Let $c \in \mathcal{M}_d$ be a parameter with super-attracting orbit and $z,z^{\prime} \in K_c$.
A closed arc $\left[z,z^{\prime}\right]$ in $K_c$ is a
\emph{regular arc} if
\begin{itemize}
\item $\left[z,z^{\prime}\right]$ has the endpoints $z$ and $z^{\prime}$,
\item for every Fatou component $U$ of $f_c$, the
intersection $\left[z,z^{\prime}\right]\cap\overline U$ is contained in the union of
at most two internal rays of $U$ together with their landing points.
\end{itemize}
We do not distinguish regular arcs which differ by re-parametrization.
\end{definition}

\begin{lemma}[Regular Arcs]
\label{Lem:RegArcs}
Let $c$ be a parameter with a super-attracting orbit. Then any two 
points $z,z^{\prime}\in K_c$ are connected by a unique regular arc in $K_c$.
\end{lemma}

\begin{proof}
For the existence of a regular arc $\left[z,z^{\prime}\right]$, take any arc in $K_c$ connecting
$z$ to $z^{\prime}$. For any bounded Fatou component $U$ of $K_c$, it suffices to assume
that $\left[z,z^{\prime}\right] \cap \overline{U}$ is connected (in fact, this is automatic). It is then
easy to modify the arc within $U$ so as to satisfy the conditions on regular arcs. 

For uniqueness, assume that there are two different regular arcs $\left[z,z^{\prime}\right]$
and $\left[z,z^{\prime}\right]^{\prime}$ for different points $z,z'\in K_c$. If
$[\,z,z'\,]\neq[\,z,z'\,]'$, then
$\mathbb{C} \setminus \left(\left[z,z^{\prime}\right]\cup\left[z,z^{\prime}\right]^{\prime}\right)$ has a bounded component $V$ (say) and
$\partial V$ is a simple closed curve formed by parts of the two regular arcs.
Since the complement of $K_c$ is connected, $V\subset K_c$, so
$V$ is contained within one bounded Fatou component $U$ of $K_c$. But now $\partial
V$ must be contained in finitely many internal rays of $U$ together with their landing
points such that two distinct external rays land at the same point: an impossibility.
\end{proof}

\begin{definition}[Hubbard Tree]
For $c \in \mathcal{M}_d$ with super-attracting orbit $\mathcal{O}$, the \emph{Hubbard tree} of $f_c$ is defined as $\bigcup_{(z,z')\in \mathcal{O} \times \mathcal{O}}[\,z,z'\,]$.
\end{definition}

It follows from Lemma~\ref{Lem:RegArcs} that every super-attracting map $f_c$ has a
unique Hubbard tree. It is easy to show that this is indeed a finite tree in the topological
sense: it has finitely many branch points and no loops.

\begin{definition}[Branch Points and Endpoints]
Let $c$ be a parameter with a super-attracting orbit $\mathcal{O}$ and associated  Hubbard
tree $\Gamma$. For $z\in\Gamma$ the components of $\Gamma \setminus \{ z \}$ are called \emph{branches of $\Gamma$ at $z$}.

If the number of branches with respect to $z$ is at least three then~$z$ is called a
\emph{branch point} of~$\Gamma$. If the number is one, then $z$ is called an
\emph{endpoint} of $\Gamma$.
\end{definition}

\begin{figure}[!ht]\label{Hubbard tree}
\begin{center}
\includegraphics[scale=0.25]{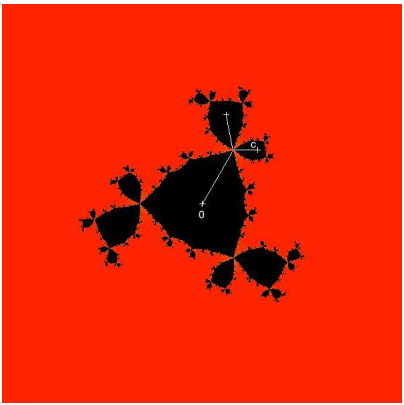}
\end{center}
\caption{The Julia set of a map $z \mapsto z^3 + c$ with a 3-periodic super-attracting
orbit. The points on the critical orbit are connected by the Hubbard tree. The tree has a branch point and the critical value is an end-point of the tree.}
\end{figure}

\begin{lemma}[The Hubbard Tree]\label{l:tree2} 
Let $c\neq 0$ be a parameter with a super-attracting orbit $\mathcal{O}$ and associated Hubbard
tree $\Gamma$. Then $\Gamma$ intersects the boundary of the Fatou component containing the critical value in exactly one point, which is periodic and
the boundary of any other bounded Fatou
component in most $d$ points which are periodic or
pre-periodic. In particular, the critical value is an endpoint of the Hubbard tree.
\end{lemma}

\begin{proof} 
Let $U_0$ be the Fatou component containing the critical point
and $U_1,\cdots,U_{n-1}$ the other bounded periodic components such that $f \left(U_i\right) = U_{i+1}$ with $U_n:=U_0$.

We denote the number of points of intersection of $\Gamma$
and $\partial U_i$ by $a_i$. For two different points $z,z'\in \mathcal{O}$ we consider their
regular arc $[\,z,z'\,]$ and pick $z^\ast \in \partial U_l\cap[\,z,z'\,]$.
Since $\Gamma$ is unique and invariant, $f_c(z^\ast)\in\partial U_{l+1}\cap\Gamma$.
This means that the set of intersection points of $\Gamma$ with the boundary of
periodic Fatou components is forward invariant. Since $f_c \colon \overline{U_l} \to \overline{U_{l+1}}$ is a one-to-one map for $l\in \{ 1,2,\ldots,n-1 \}$ and
a $d$-to-one map for $l=0$, we have $0 < a_0/d \leq a_1\leq a_2\leq \cdots \leq a_{n-1}\leq
a_0$.

We show that there is an $l^\ast\in \{ 1,2,\cdots,n-1 \}$
with $a_{l^\ast}=1$. Indeed, let $c_l:=f_c^{\circ l}(0)$ be an endpoint of
$\Gamma$. If $a_l>1$, it's easy to check that there are at least two ways to connect $c_l$ to the other points on the critical orbit by regular arcs: a contradiction. This along with the previous inequality yields the result.


\end{proof}

\begin{lemma}[Properties of Branch Points]
\label{l:tree3}
Let $c$ be a parameter with a super-attracting orbit and associated Hubbard
tree $\Gamma$. Consider a $z\in\Gamma$ such that $\Gamma$ has $m$
branches at $z$. Then:
\begin{itemize}
\item 
If $z \neq 0$ then $\Gamma$ has at least $m$ branches at $f_c(z)$.
\item 
If $z=0$ then $\Gamma$ has exactly one branch at $f_c(z)=c$.
\item 
If $z$ is a branch point then it is periodic or pre-periodic
and lies on a repelling or the super-attracting orbit.
\item
Every point on the critical orbit has at least one and at most $d$ branches.
\end{itemize}
\end{lemma}

\begin{proof}
The first statement follows from forward invariance of the Hubbard tree, and the
second was shown in Lemma \ref{l:tree2}. 

This implies the third statement because $\Gamma$ has
only finitely many branch points, and all periodic orbits other than the unique
super-attracting one are repelling.

The last claim is a restatement of Lemma \ref{l:tree2}.
\end{proof}

\subsection{Hyperbolic Components}

\begin{definition}[Roots, Co-Roots and Centers]
\label{Def:RootsAndCenters}
A \emph{hyperbolic component $H$ with period $n$} of $\mathcal{M}_d$ is a connected
component of $\{ c \in \mathcal{M}_d : f_c$ has an attracting orbit with exact
period $n \}$.
 
A \emph{root} of $H$ is a parameter on $\partial H$ with an essential parabolic orbit of exact ray period $n$ (so that the parabolic orbit disconnects the Julia set). Similarly a \emph{co-root} of $H$ is a parameter on $\partial H$ with a non-essential parabolic
orbit of exact ray period $n$ (so that the parabolic orbit does not disconnect the Julia set). A \emph{center} of $H$ is a parameter in $H$ which has a super-attracting orbit of exact
period $n$.
\end{definition}

Since our maps $f_c$ have only one critical point, we can see exactly as in the quadratic case \cite[Chapter VIII, Theorem 1.4]{CG1} that every hyperbolic component $H$ with period $n$ is a connected component of the interior of $\mathcal{M}_d$ and that there is a (non-unique) holomorphic map $z \colon H \to \mathbb{C}$ such that $z(c)$ is attracting for all $c\in H$ and has exact orbit period $n$.

Moreover, there is a holomorphic \emph{multiplier map} $\mu \colon H \to \mathbb{D}$ so
that the unique attracting orbit of $f_c$ with $c \in H$ has multiplier $\mu(c)$.
Since $\vert \mu(c) \vert \to 1$ as $c \to \partial H$, the multiplier map $\mu \colon H \to \mathbb{D}$
is a proper holomorphic map, hence extends surjectively from $\overline{H}$ to $\overline{\mathbb{D}}$ . It follows that every hyperbolic
component has (at least) one center and at least one root or co-root (in fact,
exactly one center and one root when the period of the hyperbolic component is different from one); see Theorems \ref{t:hmd} and Corollary \ref{Unique center}.
Moreover, we will see in Corollary \ref{c:par+hyp} that the number of co-roots of a
hyperbolic component of period different from one is exactly $d-2$. In particular in the quadratic
case there are no co-roots. For the hyperbolic component of period one, there are $d-1$ co-roots and no root.
\bigskip

\begin{lemma}[In the Neighborhood of Parabolic Parameters]
\label{l:rph}
Let $c_0$ be a parabolic parameter with exact parabolic orbit
period $k$, exact ray period $n$ and let $z_0$ be a point of the
parabolic orbit. Then:

\begin{itemize}
\item If the parabolic orbit portrait is non-primitive ($k\neq n$), then $c_0$ lies
on the boundary of hyperbolic components with period $k$
and $n$. Moreover, there exists a neighborhood $U$ of $c_0$ and a holomorphic
function $z_1 \colon U \to \mathbb{C}$ such that $z_1(c)$ is a point of
exact period $k$ and $z_1(c_0)=z_0$, and with the following property: for
every $c\in U \setminus \{ c_0 \}$ there is an orbit $\mathcal{O}(c)$ with exact period $n$ 
that merges into the parabolic orbit $\mathcal{O}(c_0)$ as $c \to c_0$ and for which the
multiplier map $c \mapsto \lambda \left( c,\mathcal{O}(c) \right)$ is holomorphic on $U$.

\item If the parabolic orbit portrait is primitive ($k=n$) then $c_0$ is a
root or co-root of a hyperbolic component with
period $n$. Furthermore, there are a two-sheeted cover $\pi \colon 
U'\to U$ of a neighborhood $U$ of $c_0$ with the only ramification
point $\pi(c'_0)=c_0$ and a holomorphic function $z \colon U' \to \mathbb{C}$
such that $z(c')$ is a point of exact period $n$ and $z(c'_0)=z_0$.
\end{itemize}
\end{lemma}

\begin{proof}
The proof for the quadratic case generalizes
directly to $d \geq 2$ (\cite[Lemmas 6.1 and 6.2]{M2a}
and \cite[Lemma 5.1]{S1a}). However, for completeness and because our
organization differs from the one in \cite{M2a} and \cite{S1a}, we include
the proof here.

\emph{The non-primitive case:} The multiplier of the parabolic orbit is a root of unity different from
$1$, so by the implicit function theorem there is a neighborhood $U$ of $c_0$ and a
holomorphic function~$z_1 \colon U \to \mathbb{C}$ such that $z_1(c)$ is a point with exact
orbit period $k$ and $z_1(c_0)=z_0$. This implies that the
multiplier $\lambda\bigl(c,z_1(c)\bigr)$ is a holomorphic function
in $c$ on $U$. By the Open Mapping Theorem
and $\vert\lambda\left(c_0 , z_1(c_0)\right)\vert = 1$ it follows that
every neighborhood of $c_0$ contains parameters $c$ such that $z_1(c)$
is attracting. Thus, $c_0$ lies on the boundary of a hyperbolic
component with period $k$.

We now show that it lies also on the boundary of a hyperbolic
component with period $n$: since $\lambda(f_{c_0}^{\circ n},z_0)=1$,
we obtain $f_{c_0}^{\circ
n}(z)=z+a(z-z_0)^{q+1}+O\left((z-z_0)^{q+2}\right)$ for an
integer $q \geq 1$, $a \in \mathbb{C}$, as the Taylor expansion of $f_{c_0}^{\circ
n}$ near $z_0$. It follows from the Leau-Fatou flower theorem
(\cite[Theorem~10.5]{M1new}) that $z_0$ has $q$ attracting petals and
that $f_{c_0}^{\circ n}$ is the first iterate of $f_{c_0}$ which fixes
them and the (at least $q$) dynamical rays landing at $z_0$. These rays are
permuted transitively by the first return map $f_{c_0}^{\circ k}$
(Lemma \ref{folklore}), and hence the $q$ attracting petals are also permuted
transitively.  Since $n$ is the least integer with $\lambda(f_{c_0}^{\circ
n},z_0)=1$, it follows that  $q=n/k$ and $\lambda(f_{c_0}^{\circ k},z_0)$ is an
exact $q$-th root of $1$.  It follows that there are
neighborhoods $U$ of $c_0$ and $V$ of $z_0$ such that for every $c\in U$,
$f_c^{\circ n}(z)$ has exactly $q+1$ fixed points in $V$, counted with
multiplicities. We are interested in the exact periods of these points with respect
to $f_c$ for $c\in U \setminus \{ c_0 \}$. By the above discussion exactly one of them has
exact period $k$ and no one has a lower period. Since $\lambda(f_{c_0}^{\circ
l\cdot k},z_0) \neq 1$ for $l=2,\ldots,q-1$, it follows that the iterates $f_c^{\circ
l\cdot k}$ have for $c\in U$ exactly one fixed point in $V$. 

Therefore, $q$ points in $V$ have exact period $n$, and these lie on a single orbit.
We thus have for every $c\in U \setminus \{ c_0 \}$ an orbit $\mathcal{O}(c)$ with exact period $n$
and well-defined multiplier such that $q$ points of $\mathcal{O}(c)$ each coalesce at one
point of the parabolic orbit $\mathcal{O}(c_0)$ as $c\to c_0$. The
multiplier $c \mapsto \lambda \left(c,\mathcal{O}(c)\right)$ defines a holomorphic
function on $U$ (One cannot, in general, follow the individual points of the period $n$ orbit holomorphically. What one can rather follow holomorphically are symmetric functions of the points on the periodic orbit; the multiplier is indeed a symmetric function of the periodic points and it extends holomorphically to $c$ by Riemann's removable singularity theorem). As before it follows by the Open Mapping Theorem
that $c_0$ lies on the boundary of a hyperbolic component with
period $n$.

\emph{The primitive case:} In the primitive case we have again $\lambda(f_{c_0}^{\circ
n},z_0)=1$ and therefore the Taylor expansion $f_{c_0}^{\circ
n}(z)=z+a(z-z_0)^{q+1}+ O\bigl((z-z_0)^{q+2}\bigr)$ near $z_0$ for an
integer $q \geq 1$, $ a \in \mathbb{C}$. Since each of the $q$ petals must absorb a
critical orbit and $f_{c_0}$ has only one, we see $q=1$ and the multiplicity of $z_0$ as a root of $f_{c_0}^{\circ
n}$ is exactly 2. Therefore, it splits into two simple fixed points of $f_{c_0}^{\circ
n}$ when $c_0$ is perturbed. These two fixed points have exact period $n$ under the original map $f_{c_0}$.

As $c$ traverses a small loop around $c_0$, these two fixed
points are interchanged. But they are at their original positions after two
loops around $c_0$. Hence, on a two-sheeted cover $U'$ (let the projection map be $\pi$ so that $\pi$ is branched only over $c_0$) of a neighborhood of $c_0$,
the two fixed points can be defined as the values of two holomorphic
functions $z_1(c'),z_2(c')$ with corresponding holomorphic multipliers $\lambda_1(c'),\lambda_2(c')$
(initially, the functions $z_i(c')$ are defined on a two-sheeted cover of a punctured neighborhood
of $c_0$, but the puncture can then be filled in). Since we have $\lambda_i(\pi^{-1}(c_0))=+1$, 
the Open Mapping Theorem implies that there is a parameter $c_1^{\prime} \in U'$ with $c_1 = \pi(c_1^{\prime}) \in \mathcal{M}_d$ such that $\vert \lambda_1(c_1^{\prime}) \vert < 1$. This implies that there is a point $c_1$ arbitrarily close to $c_0$ with an attracting periodic orbit of exact period $n$. Thus, $c_0$ lies on the boundary of a
hyperbolic component of period $n$. This proves the lemma.
\end{proof}

\subsection{Orbit Separation Lemmas} \label{ss:osl}
Now we establish a couple of orbit separation lemmas which will be useful in the sequel. They show the existence of pairs
of dynamical rays landing at a common repelling periodic or pre-periodic point which
separate certain points within the filled-in Julia set, in the following sense:
two points $z,z'\in K_c$ are \emph{separated} in the dynamical plane if there are
two dynamical rays $\mathcal{R}_\theta^{c}, \mathcal{R}_{\theta'}^{c}$ landing at a common repelling
point $z_0$ such that $z$ and $z'$ are in different components
of $\mathbb{C} \setminus \left(\mathcal{R}_\theta^{c}\cup \mathcal{R}_{\theta'}^{c}\cup \{ z_0 \} \right)$. It is convenient to call these two rays together with their landing point
the \emph{ray pair} at angles $(\theta,\theta')$. One useful feature is that such a co-landing
ray pair may be stable even small perturbation when the landing behavior of other dynamical rays is not.

\begin{lemma}[Orbit Separation in Super-attracting Case]
\label{Lem:OrbitSeparationSuperattr}

Suppose that $f_c$ has a super-attractive orbit of some period $n$. Then for
every two repelling periodic points $z$ and $z'$ on the same orbit with non-trivial orbit
portrait, there exists a repelling periodic or pre-periodic point $w$ on a different
orbit so that two periodic dynamical rays landing at $w$ separate $z$ from $z'$.
\end{lemma}

\begin{proof}
By Lemma \ref{CharPoint}, the characteristic point on the orbit of $z$ is on
the Hubbard tree $\Gamma$, hence the entire orbit of $z$ is. Let $\Gamma'$ be the
union of regular arcs between the points on the orbit of $z$; clearly
$\Gamma' \subset \Gamma$. Let $\Gamma'' \subset \Gamma$ be the component of $\Gamma \setminus
\{ f_c^{-1}(z) \}$ containing the critical point. Then any regular arc in
$\Gamma' \setminus \Gamma''$ has its $f_c$-image in $\Gamma'$, so any regular arc in
$\Gamma'$ has its $f_c$-image in $\Gamma'\cup[z,c]$ (where $c$ is the critical
value). 

Suppose first that $z$ is the characteristic point of its orbit; then $z$ is an
endpoint of the tree $\Gamma'$ (similarly as in Lemma \ref{l:tree2}). We may suppose
that $\left( z,z' \right)$ does not contain any point on the orbit of $z$. If $\left( z,z' \right)$ contains a
branch point $w$ of $\Gamma'$, then all points on the forward orbit of $w$ have at
least as many branches in $\Gamma'$ as $w$, except when the orbit runs through $z$;
in the latter case, the orbit may lose one branch (the branch to the critical
value $c$), but $z$ is in fact an endpoint of $\Gamma'$ and thus cannot be on the orbit
of $w$. Therefore, this branch point $w$ must be (pre)periodic. If $w$ is a repelling (pre)periodic point, then it must be the landing point of at least two rational dynamical rays and we are done. If $w$ lies on the critical orbit (let $U$ be the periodic Fatou component containing $w$), then there must be at least one repelling (pre)periodic point $w'$ other than $z$ and $z'$ on $\left[ z,z' \right] \cap \partial U$ separating $z$ from $z'$. As before, $w'$ must be the landing point of at least two rational dynamical rays. Thus the lemma is proved if $(z,z')$ contains a branch point of $\Gamma'$. 

We now assume that $\left( z,z' \right)$ contains no branch point of the tree, and no point on the
orbit of $z$. Let $n$ be the period of $z$, and let $k \in \{ 1,2,\cdots,n-1 \}$ be
so that $f^{\circ k}(z')=z$. If $f^{\circ k}$ is injective on $[z,z']$, then $\left[ z,z' \right]$ must contain a (repelling) fixed point of $f^{\circ k}$ and the lemma is proved in this case.

Otherwise, there is a minimal $k' < k < n$ such that $f_c^{\circ k'}\left( \left[ z,z' \right] \right) \ni 0$. Then
there is a point $z''\in \left[ z,z' \right]$ such that $f^{\circ (k'+1)} \colon [z'',z'] \to
\left[ c,f^{\circ (k'+1)}(z')\right] \supset [z',z]$ is a homeomorphism, and again there is a
point $w \in (z,z')$ which is fixed under $f^{\circ (k'+1)}$ (necessarily repelling), so this case is done as
well.

In order to treat the case that $z$ is not the characteristic point of its orbit,
it suffices to iterate $f_c$ until it brings $z$ to the characteristic point.
\end{proof}
 
We will need the concept of parabolic trees, which are defined in analogy with Hubbard trees for post-critically finite polynomials. Our definition will follow \cite[Section 5]{HS}. The proofs of the basic properties of the tree can be found in \cite[Lemma 3.5, Lemma 3.6]{S1a}.

\begin{definition}[Parabolic Tree]
If $c$ lies on a parabolic root arc of period $k$, we define a \emph{loose parabolic tree} of $f_c$ as a minimal tree within the filled-in Julia set that connects the parabolic orbit and the critical orbit, so that it intersects the critical value Fatou component along a simple $f_c^{\circ k}$-invariant curve connecting the critical value to the characteristic parabolic point, and it intersects any other Fatou component along a simple curve that is an iterated pre-image of the curve in the critical value Fatou component. Since the filled-in Julia set of a parabolic polynomial is locally connected and hence path connected, any loose parabolic tree connecting the parabolic orbit is uniquely defined up to homotopies within bounded Fatou components. It is easy to see that any loose parabolic tree intersects the Julia set in a Cantor set, and these points of intersection are the same for any loose tree (note that for simple parabolics, any two periodic Fatou components have disjoint closures).
\end{definition}

By construction, the forward image of a loose parabolic tree is again a loose parabolic tree. A simple standard argument (analogous to the post-critically finite case) shows that the boundary of the critical value Fatou component intersects the tree at exactly one point (the characteristic parabolic point), and the boundary of any other bounded Fatou component meets the tree in at most $d$ points, which are iterated pre-images of the characteristic parabolic point \cite[Lemma 3.5]{S1a}. The critical value is an endpoint of any loose parabolic tree. All branch points of a loose parabolic tree are either in bounded Fatou components or repelling (pre-)periodic points; in particular, no parabolic point (of odd period) is a branch point.

Following \cite[\S 3]{S1a}, we now define a preferred parabolic tree as follows: Let $U$ be the critical value Fatou component of $f_c$, and let $w$ be the characteristic parabolic periodic point. First we want to connect the critical value $c$ in $U$ to $w$ by a simple curve which
is forward invariant under the dynamics. We will use Fatou coordinates for the attracting
petal of the dynamics \cite[\S 10]{M1new}. In these coordinates, the dynamics is simply addition
by $+1$, and our curve will just be the pre-image under the Fatou coordinate of a horizontal straight line in a right half-plane connecting the images of the critical orbit. Since any bounded Fatou component eventually maps onto the critical value Fatou component, we now require the parabolic tree in any other bounded Fatou component to be a pre-image of this chosen curve. With this choice, we have specified a preferred tree which is invariant under the dynamics. We will refer to this tree as \emph{the} parabolic tree.

\begin{lemma}[Orbit Separation Lemma For Two Parabolic Points]
\label{Lem:OrbitSeparationParabolic}
Suppose that $f_{c}$ has a parabolic orbit. Then for any two parabolic periodic
points $z \neq z'$, there exists a ray pair landing at a repelling periodic or
pre-periodic point which separates $z'$ from $z$.
\end{lemma}

\begin{proof}
It suffices to prove the lemma when $z$ is the characteristic point of the parabolic cycle (otherwise we can iterate $f_c$ until $z$ satisfies this condition). We may assume that the part of the parabolic tree between $z$ and $z'$ neither traverses a periodic Fatou component except at its ends, nor does it traverse another parabolic periodic point. If there is a branch point of the tree between $z$ and $z'$, this branch point must be the landing point of two rational rays separating the orbit and we are done.

Otherwise, we argue as in the previous lemma: let $n$ be the period of $z$, and let $k \in \{ 1,2,\cdots,n-1 \}$ be so that $f^{\circ k}(z')=z$. Put $f^{\circ k}(z)=z''$, then $f^{\circ k}$ maps $\left[ z',z \right]$ onto $\left[ z ,z'' \right] \supseteq \left[ z ,z' \right]$. Since $\left[ z',z \right]$ contains no branch point of the tree, $\left[ z ,z'' \right]$ cannot branch off as well; which implies that there exists $z^{\ast} \in \left( z',z \right)$ fixed by $f^{\circ k}$. Clearly, $z^{\ast}$ is a repelling periodic point disconnecting the Julia set and hence is the landing point of two rational dynamical rays separating the orbit. This completes the proof of the lemma.
\end{proof}

\subsection{Results}\label{ss:osldyn}
Now we can show that at least certain rays land pairwise and give a complete description of the periodic rays landing at primitive parabolic parameters.

\begin{theorem}[A Necessary Condition]\label{t:perchar}
If a parameter ray $\mathcal{R}_\theta$ at a periodic angle lands at a parameter $c_0$, then the
dynamical ray $\mathcal{R}_\theta^{c_0}$ lands at the characteristic point of the parabolic
orbit of $c_0$.
\end{theorem}

\begin{proof}
The landing point $c_0$ of $\mathcal{R}_\theta$ is necessarily parabolic by Theorem \ref{per-rays-land}, and 
the dynamical ray $\mathcal{R}_\theta^{c_0}$ lands at a point of the parabolic orbit of $c_0$.  Without restriction we assume that the exact parabolic orbit period is at least $2$.

Let, $z$ be the characteristic point on the parabolic cycle. We assume that the dynamical ray $\mathcal{R}_\theta^{c_0}$ lands at some point $z'$ of the parabolic orbit where $z' \neq z$. Then the Orbit Separation Lemma \ref{Lem:OrbitSeparationParabolic} shows that there is a rational ray pair at
angles $\theta_1,\theta_2$ landing at some common (pre)periodic repelling point $w$ separating $z$ from $z'$.  In particular, the dynamical ray $\mathcal{R}_\theta^{c_0}$ and the critical value $c_0$ belong to two different regions of the partition $\mathbb{C} \setminus (\mathcal{R}_{\theta_1}^{c_0} \cup \mathcal{R}_{\theta_2}^{c_0})$. For all parameters $c$ close to $c_0$, the two parameter rays at angles $\theta_1$ and $\theta_2$ continue to land together \cite[Lemma~B.1]{GM1} and the partition is stable \cite[Lemma 2.2]{S1a} in following sense: the dynamical ray $\mathcal{R}_\theta^{c}$ and the critical value $c$ belong to different regions of this partition in the dynamical plane of $f_c$. But $c_0$ is a limit point of the parameter ray $\mathcal{R}_\theta$ and there are parameters $c$ arbitrarily close to $c_0$ such that in the dynamical plane of $f_c$, the critical value $c$ lies on the dynamical ray $\mathcal{R}_\theta^{c}$; a contradiction. This proves that the landing point of the dynamical ray $\mathcal{R}_\theta^{c_0}$ must be the characteristic point $z$ on the parabolic orbit.
\end{proof}

The next result shows that any parabolic parameter with a non-trivial orbit portrait is the landing point of at least two rational parameter rays.

\begin{lemma}[At Least the Characteristic Rays Land at a Parameter]
\label{At least two rays}
Every parabolic parameter with a non-trivial portrait is the landing point of the parameter rays
at the characteristic angles of the parabolic orbit portrait.
\end{lemma}

\begin{proof}
Let $c$ be a parabolic parameter and denote its parabolic orbit portrait by
$\mathcal{P} = \{ \mathcal{A}_1,\cdots, \mathcal{A}_{p} \}$. Let $t^-$ and $t^+$ be the
characteristic angles and label the elements of $\mathcal{P}$ cyclically so that
$t^-, t^+ \in \mathcal{A}_1$.

By Theorem \ref{t:p-wake}, the two parameter rays $\mathcal{R}_{t^-}$ and
$\mathcal{R}_{t^+}$ land at a common parabolic parameter $c'$ (say) with associated wake $\emph{W}$ and let
$\mathcal{P}' = \{ \mathcal{A}_1^{\prime},\cdots, \mathcal{A}_{q}^{\prime} \}$ be the orbit portrait of its parabolic orbit. By Theorem \ref{t:perchar}, the dynamical rays $\mathcal{R}_{t^-}^{c'}$ and
$\mathcal{R}_{t^+}^{c'}$ land at the same point of the parabolic orbit. There is thus an
element in $\mathcal{P}'$ which contains both $t^-$ and $t^+$; call this element
$\mathcal{A}_1^{\prime}$. It follows from Lemma \ref{Lem:Characteristic angles} that $ \mathcal{P}= \mathcal{P}'$. Therefore, the parabolic parameters $c$ and $c'$ have the same parabolic orbit portrait $\mathcal{P}$.

By \cite[Theorem 4.1]{M2a} (the quadratic case easily generalizes to unicritical polynomials of any degree), both $c$ and $c'$ are limit points of parameters with a repelling periodic orbit with associated orbit portrait $\mathcal{P}$. Now Lemma \ref{c:wake} tells us that $c, c' \in \partial W$. This clearly implies that $c = c'$ and $c$ is the landing point of the parameter rays at the characteristic angles of the parabolic orbit portrait.
\end{proof}

\begin{lemma}[Parabolic Parameters with Trivial Portrait]
\label{trivial portraits}
Let $c_0$ be a parabolic parameter with a trivial orbit portrait. Then at least one parameter
ray lands at $c_0$.
\end{lemma}

\begin{proof}
We continue with the proof of the Lemma \ref{l:rph}. Let $f_{c_0}$ have a parabolic orbit of period $n$ and $z_0$ be a representative point of the parabolic orbit. Then $z_0$ is a fixed point of multiplicity 2 for the map $f_{c_0}^{\circ n}$ and there exists a two-sheeted cover $U'$ (with a projection map $\pi$ so that $\pi$ is branched only over $c_0$) of a neighborhood of $c_0$ such that
the two simple fixed points of the perturbed maps can be defined as the values of two holomorphic
functions $z_1(c'),z_2(c')$ with corresponding holomorphic multipliers $\lambda_1(c'),\lambda_2(c')$. Since we have $\lambda_i(\pi^{-1}(c_0)) = +1$, the Open Mapping Theorem implies that there is a
parameter $c_1^{\prime} \in U'$ such that $c_1 = \pi (c_1^{\prime}) \in \mathcal{M}_d$ and  $\vert \lambda_1(c_1^{\prime}) \vert < 1$ and therefore
$\vert \lambda_2(c_1^{\prime}) \vert > 1$ (there can be at most one non-repelling orbit). 
Since $z_2(c_1^{\prime})$ is repelling, it is the landing point of at least one periodic
dynamical ray (Lemma \ref{Lem:Landing of Dyn Rays}); so the corresponding orbit has a
portrait $\mathcal{P} \neq \emptyset$. Similarly, there is another parameter $c_2^{\prime} \in
U'$ such that $c_2 = \pi (c_2^{\prime}) \in \mathcal{M}_d$ and $\vert \lambda_2(c_2^{\prime}) \vert < 1$; so no periodic dynamical ray lands at $z_2(c_2^{\prime})$. We may assume that $U'$ is small enough so that it does not contain parabolic parameters other than $c_0$ of the same ray
period $n$. Then Lemma~\ref{l:preserving-portrait1} implies that any path connecting $c_1$ and
$c_2$ must cross a parameter ray at an angle of period $n$, which lands at $c_0$. 
\end{proof}

Now we can prove the Structure Theorem \ref{t:struct} in the primitive
case, in particular Statement 2:

\begin{corollary}[Parameter Rays Landing at Primitive Parameters]
\label{c:nonesspar} 
Let $c$ be a primitive parabolic parameter. Then $c$ is the landing point of the parameter rays at precisely those angles $\theta$ such that the corresponding dynamical rays $\mathcal{R}_{\theta}^c$ land at the characteristic parabolic point of $f_c$. The number of such rays is exactly two if
the parabolic orbit is essential, and exactly one otherwise.
\end{corollary}

\begin{proof}
We know by Lemma \ref{At least two rays} and Lemma \ref{trivial portraits}: at least two parameter
rays at periodic angles land at $c$ if $f_c$ has a non-trivial parabolic orbit portrait, and at least one parameter ray at a periodic angle lands at $c$ if the associated parabolic portrait is trivial.

By Theorem \ref{t:perchar}, only those parameter rays $\mathcal{R}_\theta$ can land at
$c$ for which the dynamical ray $\mathcal{R}_\theta^{c}$ lands at the characteristic point
of the parabolic orbit and there are only two (resp. one) such rays available in the primitive case. 
\end{proof}

The state of affairs is now as follows: every parameter ray at a periodic angle lands at a parabolic parameter. Conversely, given a primitive parabolic parameter $c$, we have shown that it is the landing point of \emph{exactly} two or one parameter ray(s) at periodic angle(s) depending on whether the parabolic orbit portrait is non-trivial or not. On the other hand, if $c$ is a non-primitive (satellite) parabolic parameter, then it is the landing point of \emph{at least} two parameter rays at the characteristic angles of the associated parabolic orbit portrait. Thus it remains to show that a parabolic parameter of satellite type is the landing point of \emph{at most} two parameter rays at periodic angles. The following two sections will be devoted to the proof of this statement.

\section{Roots and Co-Roots of Hyperbolic Components}\label{s:hyp2}
The goal of this section is to show that every hyperbolic component of period different from one 
has exactly one root and exactly $d-2$ co-roots.

\begin{theorem}[Continuous Dependence of Landing Points on Parameters]
\label{t:cpl} 
Let $z_0$ be a repelling or parabolic periodic point of $f_{c_0}$.
For a dynamical ray $\mathcal{R}_\theta^{c_0}$ landing at $z_0$,
let $\Omega(\theta):= \{ c \in \mathbb{C} : \mathcal{R}_\theta^{c}$ lands $\}$.  Then
there is a continuous $z \colon \Omega(\theta)\to \mathbb{C}$ such that $z(c)$ is
the landing point of $\mathcal{R}_{\theta}^{c}$.
\end{theorem}

\emph{Remark.}
Although the landing point of the dynamical ray depends continuously on
the parameter $c$ (if the ray lands), the portrait may be
destroyed. This is certainly always the case whenever the orbits of the landing points have different periods. For example it may occur that $z(c)$ splits into several periodic points while perturbing away from a parabolic parameter
, among which the rays of the parabolic point are distributed.

\begin{proof}
The proof is analogous to the one in the quadratic case, see \cite[Proposition 5.1]{S1a}.
\end{proof}

\begin{corollary}[Stability of Portraits in a Hyperbolic Component
and its Roots and Co-Roots]
\label{c:phe}
Let~$H$ be a hyperbolic component and~$E$ be the set of all roots and
co-roots of~$H$. If $c_0 \in H\cup E$ and $z_0$ be a
repelling or parabolic periodic point of $f_{c_0}$, then there is a
continuous map $z \colon H\cup E \to \mathbb{C}$ such that $z(c_0)=z_0$ and the
portrait of the orbit of $z(c)$ is the same for all $c\in H\cup E$.
\end{corollary}

\begin{proof}
Note that for all $c \in H\cup E$, all periodic dynamical rays
land. Hence, their landing points depend continuously on the
parameter by Theorem \ref{t:cpl}. In order to prove the corollary, it is enough to show the following: if the dynamical rays, say~$\mathcal{R}_\theta^{c_0}$ and~$\mathcal{R}_{\theta'}^{c_0}$ land at~$z_0$ for some $c_0 \in
H \cup E$, then ~$\mathcal{R}_\theta^c$ and~$\mathcal{R}_{\theta'}^c$ land together for all~$c\in H \cup E$.

Let $z(c)$ and $z^{\prime}(c)$ be two continuous functions on $H \cup E$ such that $z(c_0) = z^{\prime}(c_0) = z_0$ and $z(c)$ (resp. $z^{\prime}(c)$) is the landing point of $\mathcal{R}_\theta^c$ (resp. $\mathcal{R}_{\theta'}^c$). 

If $z_0$ is a repelling point, then $\mathcal{R}_\theta^{c}$ and~$\mathcal{R}_{\theta'}^{c}$ continue to land together at repelling periodic points for all $c$ close to $c_0$. By Lemma \ref{l:preserving-portrait1}, they co-land throughout $H$; i.e. $z(c) = z^{\prime}(c)$ for all $c$ in $H$. By continuity, $z(c) = z^{\prime}(c)$ for all $c$ in $E$ and we are done.

Now suppose that $z_0$ is a parabolic point. Let~$k$ be the orbit period and~$n$ the ray period of~$z_0$. By Lemma~\ref{l:rph}, points of a~$k$-periodic and an~$n$-periodic orbit coalesce at~$z_0$ and no further orbits are involved. Since one of them is attracting (namely the~$n$-periodic orbit in the non-primitive case), $z(c)$ and $z^{\prime}(c)$ always have period~$k$. Since there is only one orbit of period $k$ available, it follows that $z(c) = z^{\prime}(c)$ for $c$ close to $c_0$ with $c \in H$. As above, Lemma \ref{l:preserving-portrait1} implies that $z(c) = z^{\prime}(c)$ for $c \in H$. By continuity, the same holds at the roots and co-roots. Therefore, $\mathcal{R}_\theta^c$ and~$\mathcal{R}_{\theta'}^c$ land together for all~$c\in H \cup E$.
\end{proof}

\begin{definition}[Multiplier Map of a Hyperbolic Component]
Let~$H$ be a hyperbolic component and for~$c \in H$
let~$\lambda(c,\mathcal{O})$ be the multiplier of the unique attracting orbit
of~$f_c$. Then the map $\lambda_H \colon H \to \mathbb{D}, \hspace{1mm} c \mapsto \lambda(c,\mathcal{O})$ is called the \emph{multiplier map
of~$H$}.
\end{definition}

\noindent The multiplier map~$\lambda_H$ is well-defined: for
hyperbolic parameters there is a unique attracting orbit and the
absolute value of the multiplier of an attracting orbit is less than
one. Precisely as in the quadratic case we see that the multiplier
map~$\lambda_{H}$ of a hyperbolic component~$H$ is a proper
holomorphic map and has a continuous extension~$\lambda_{\overline{H}}$ from~$\overline{H}$ onto~$\overline{\mathbb{D}}$.

\begin{lemma}
Let~$H$ be a hyperbolic component of period $n$ with center~$c_0$ and set of roots and co-roots $E$ (a subset of $\partial H$). For any $c \in H \cup E$, there are exactly $(d-1)$ points on the boundary of the characteristic Fatou component $U_1$ of $f_c$ which are fixed by the first return map of $U_1$. All these periodic points are repelling or parabolic.
\end{lemma}

\begin{proof}
Note that for all $c \in H \cup E$, the Julia set $J(f_c)$ is locally connected. Since $U_1$ is simply connected, there exists a Riemann map $\phi \colon U_1 \to \mathbb{D}$ which extends to a homeomorphism of the closures (this is implied by the local connectivity of $J(f_c)$ and the fact that the filled-in Julia set of $f_c$ is full). Then $\phi$ conjugates the first return map $f_{c}^{\circ n}$ of $U_1$ to a proper degree $d$ holomorphic self-map of $\mathbb{D}$, hence a Blaschke product of degree $d$, say $\mathcal{B}_c$. As conjugate dynamical systems have the same number of fixed points, the number of points on the boundary of $U_1$ which are fixed by the first return map of $U_1$ is equal to the number of fixed points of $\mathcal{B}_c$ on $\partial \mathbb{D}$.

In the (super-)attracting case, there is exactly one fixed point of $\mathcal{B}_c$ in $\mathbb{D}$ and by reflection, exactly one fixed point in $\mathbb{C} \setminus \overline{\mathbb{D}}$. Since a rational map of degree $d$ has $(d+1)$ fixed points counted with multiplicity, there must be $(d-1)$ fixed points (counted with multiplicity) on $\partial \mathbb{D}$. Since these fixed points are never parabolic, each of them has multiplicity 1; i.e. there are exactly $(d-1)$ fixed points of $\mathcal{B}_c$ on $\partial \mathbb{D}$. In the parabolic case, there are no fixed points of $\mathcal{B}_c$ in $\mathbb{D} \cup \left( \mathbb{C} \setminus \overline{\mathbb{D}} \right)$. So all the  
$(d+1)$ fixed points lie on $\partial \mathbb{D}$. $\mathcal{B}_c$ has a parabolic fixed point on $\partial \mathbb{D}$ and the Julia set is all of $\partial \mathbb{D}$. Clearly, there are two attracting petals; i.e. the parabolic fixed point has multiplicity 3 and there are exactly $(d-2)$ simple fixed points, all distinct. Hence, the total number of distinct fixed points of $\mathcal{B}_c$ on $\partial \mathbb{D}$ is again $(d-1)$. 

Therefore, $U_1$ has exactly $(d-1)$ points on its boundary which are fixed by its first return map. Since there can be only one non-repelling periodic orbit of~$c$, these periodic points are either all repelling or exactly one of them is parabolic.
\end{proof}

\begin{lemma}[On the Boundary of the Characteristic Fatou Component]
\label{l:brycf}
Let~$H$ be a hyperbolic component of period $n$ with center~$c_0$ and set of roots and co-roots $E \left( \subset \partial H \right)$. There are $(d-1)$ continuous functions $z^{(1)},\cdots,z^{(d-1)}$ on $H \cup E$ such that for any $c \in H \cup E$, $\{ z^{(1)}(c),\cdots,z^{(d-1)}(c)\}$ are precisely the~$\left( d-1 \right)$ points on the boundary of the characteristic Fatou component $U_1$ which are fixed by the first return map of $U_1$. At exactly one of them, more than one dynamical rays land. Moreover, at every~$c \in E$, one of the~$z^{(i)}(c)$'s is the characteristic point of
the parabolic orbit.
\end{lemma}

\begin{proof}
The super-attracting point $c_0$ in the dynamical plane of $f_{c_0}$ can be holomorphically followed throughout $H$ yielding an analytic function $z^{\ast} \colon H \to \mathbb{C}$ with ~$z^\ast(c_0)=c_0$ and $z^{\ast}(c)$ periodic of period $n$ for all $c$ in $H$. $z^{\ast}$ can be extended to a continuous function on~$H \cup E$ and since the multiplier map~$\lambda_{\overline{H}}$ is
proper holomorphic on~$\overline H$, $z^\ast(c)$ must have multiplier 1 for every~$c\in E$. Also, $z^\ast(c)$ lies in the closure of the critical value Fatou component for every $c \in H \cup E$, so it must be the characteristic point of the parabolic orbit for $c \in E$. 

Let the $(d-1)$ points on the boundary of the characteristic Fatou component $U_1$ of $f_{c_0}$ which are fixed by the first return map of $U_1$ be $z_0^{(1)},\ldots,z_0^{(d-1)}$. Since the~$z_0^{(i)}$'s are repelling, by Corollary~\ref{l:preserving-portrait1} and \ref{c:phe} there are continuous functions~$z^{(i)}$ on~$H \cup E$ with~$z^{(i)}(c_0)=z_0^{(i)}$ for~$1 \leq i \leq d-1$ and the~$z^{(i)}(c)$'s
are repelling for~$c \in H$ such that for any fixed $i$, the portrait of the orbit of $z^{(i)}(c)$ remains constant for all $c \in H \cup E$. Also, each $z^{(i)}(c)$ lies on the boundary of the characteristic Fatou component of $f_{c}$ and is fixed by the first return map of the component. Since~$z^\ast(c)$ has period $n$ and lies on the boundary of the characteristic Fatou component of~$c$ for~$c \in E$, it must be one of the points~$z^{(i)}(c)$'s.

Finally we show that exactly one of the $z^{(i)}(c)$'s is the landing point of more than one rays. We first prove the super-attracting case following the proof of \cite[Lemma 3.4]{NS}: without restriction we assume~$n>1$ (If~$n=1,$ the only parameter with a super-attracting fixed point is $0$ and the dynamical rays at angles~$0$ and~$1$, which we consider as two different rays in this case, land trivially at a common point in the dynamical plane of $z^d$). One of the~$z_0^{(i)}$'s, say~$z_0^{(1)}$, lies on the Hubbard tree~$\Gamma$ of~$c_0$ and it is the only point of $\Gamma\cap\overline{U_1}$ by Lemma~\ref{l:tree2}. Therefore,~$z_0^{(1)}$ disconnects~$\partial K_c$ and is the landing point of at least two dynamical rays. If there is a second $z_0^{(i)}$ $(i \neq 1)$ which disconnects the filled-in Julia set of $c_0$, then there must be a periodic or pre-periodic Fatou component $U^{\prime}$ which is separated from the characteristic Fatou component by $z_0^{(i)}$. Let $\gamma$ be an injective curve in the filled-in Julia set which connects the critical value to the center of $U^{\prime}$ (that
is the unique point of $U^{\prime}$ which maps onto the critical orbit); $\gamma$ becomes unique if we require that it maps onto the Hubbard tree by the time that $U^{\prime}$ lands on a periodic component. From then on, all forward iterates of $\gamma$ will be on the Hubbard tree, so they will not meet $z_0^{(i)}$; this is in contradiction to periodicity of $z_0^{(i)}$ and shows that exactly one of the $z_0^{(i)}$s can disconnect the filled-in Julia set of $c_0$; i.e. exactly one of them (say $z_0^{(1)}$) is the landing point of more than one rays. By Corollary \ref{c:phe}, only $z^{(1)}(c)$ must be the landing point of more than one dynamical rays for each $c \in H \cup E$.
\end{proof}

\begin{theorem}[Mapping Degree of~$\lambda_{\overline{H}}$]
\label{t:hmd}
The multiplier map~$\lambda_{\overline{H}}$ has mapping degree~$d-1$ and every hyperbolic component $H$ of period greater than one has exactly one root.
\end{theorem}

\begin{proof}
Let~$H$ be a hyperbolic component with period~$n$. The mapping
degree of~$\lambda_{\overline H}$ is at least~$d-1$: for a
center~$c_0$ of~$H$ there is a holomorphic function~$z(c)$ on $H$ such
that~$z(c_0)=0$ and we can locally write $\displaystyle$ $\lambda_{\overline{H}}(c)$ $ = d^n$ $f_c^{\circ(n-1)}$ $\left(z(c)\right)^{d-1}$ $\cdots
f_c\left(z(c)\right)^{d-1}$ $z(c)^{d-1}$. Since~$z(c)$ has the only
zero at~$c_0$, it follows that~$\displaystyle$ $\lambda_{\overline{H}}(c)=d^n\left(c-c_0\right)^{d-1}g(c)^{d-1}$ for some holomorphic function~$g$
that does not vanish in a neighborhood of~$c_0$, and~$\lambda_{\overline{H}}$ has mapping
degree at least ~$d-1$.

Note that the map $\lambda_{\overline{H}}$ takes the value $1$ precisely at the roots and co-roots of $H$. Therefore, the mapping degree of $\lambda_{\overline{H}}$ is bounded above by the total number of roots and co-roots. By the previous Lemma~\ref{l:brycf}, exactly one of the $z^{(i)}$'s becomes the characteristic point of the parabolic orbit at each root or co-root. Fix $ i \in \{ 1,2,\cdots,d-1 \}$; we show that there exists a unique $c \in E$ for which $z^{(i)}(c)$ is the characteristic point of the parabolic orbit. Suppose there was another $c^{\prime}$ with the same property. Let the orbit portrait of the orbit of $z^{(i)}$ be $\mathcal{P}_i$ which remains constant throughout $H \cup E$ and let $\theta$ be a characteristic angle of $\mathcal{P}_i$ (if $\mathcal{P}_i$ is trivial, then $\theta$ is the angle of the only dynamical ray landing at the characteristic point $z^{(i)}$). By Lemma \ref{At least two rays} and Corollary \ref{c:nonesspar}, the parameter ray at angle $\theta$ lands both at $c$ and at $c^{\prime}$; a contradiction which proves our claim. Therefore, there are only~$d-1$ candidates for parabolic parameters with ray period~$n$ on the boundary of $H$ implying that the total number of roots and co-roots of $H$ is at most $d-1$. Thus, the mapping degree of~$\lambda_{\overline{H}}$ is at
most~$d-1$ and hence precisely~$d-1$.

Moreover, this shows that all candidates for characteristic points are
realized. Since portraits are stable for all parameters in~$H \cup E$ (Corollary~\ref{c:phe}) we obtain by Lemma~\ref{l:brycf} that exactly one parameter in $E$ has a parabolic orbit with a non-trivial orbit portrait (when the period is different from one). This proves that each hyperbolic component of period different from one has exactly one root.
\end{proof}

\begin{corollary}[Number of Co-Roots]
\label{c:par+hyp}
Every hyperbolic component of period greater than one has exactly~$d-2$ co-roots.
\end{corollary}

\noindent
The previous two statements, Theorem~\ref{t:hmd} and Corollary~\ref{c:par+hyp}, prove the last assertion of
the Structure Theorem.

\begin{corollary}\label{Unique center}
Every hyperbolic component of $\mathcal{M}_d$ has exactly one center.
\end{corollary}

\begin{proof}
By Theorem \ref{t:hmd} and its proof, $H$ has at least one center $c_0$ (i.e. $\lambda_{\overline{H}}(c_0)=0$) such that the local mapping degree of $\lambda_{\overline{H}}$ at $c_0$ is $d-1$. But, the mapping degree of $\lambda_{\overline{H}}$ is $d-1$, and hence, $c_0$ is the unique parameter where $\lambda_{\overline{H}}$ vanishes. Therefore, $H$ has a unique center.
\end{proof}

So far we have showed that \emph{at least}~$d$
parameter rays at $n$-periodic angles land on the boundary of every hyperbolic component of the same period $n$. For the proof of the Structure Theorem it remains
to show that \emph{at most}~$d$ parameter rays land at every
hyperbolic component. For this purpose we need to connect the landing points of the various
parameter rays at periodic angles landing on the boundary of a common hyperbolic component by \emph{internal rays}.

\begin{definition}[Internal Rays of a Hyperbolic Component]
An \emph{internal ray} of a hyperbolic component~$H$ is an
arc~$c \colon \left[0,1\right] \to \overline{H}$ starting at the center such that
there is an angle~$\theta$ with~$\displaystyle\lambda_{\overline{H}}(c) = \left[0,1\right] \cdot e^{2\pi i\theta}$.
\end{definition}

\emph{Remark.}
Since~$\lambda_{H}$ is a~$(d-1)$-to-one map, an internal ray of~$H$
with a given angle is not uniquely defined. In fact, a hyperbolic component has~$\left( d-1 \right)$ internal rays with
any given angle~$\theta$.

\section{Kneading Sequences}\label{s:kneading}
In this section we complete the description of the landing properties
of parameter rays at periodic angles by an induction proof on the ray period,
Theorem~\ref{t:exactlytwo}. We use a similar strategy as
in~\cite[Section~3]{S1a}. However, contrary to the quadratic case,
we need for~$d > 2$ some knowledge on the hyperbolic components, which we
accumulated in the previous sections.

\begin{definition}[Itineraries and Kneading Sequences]
Fix $d \geq 2 $. For an angle~$\theta \in \mathbb{R}/\mathbb{Z}$ label the components of $\mathbb{R}/\mathbb{Z} \setminus \{ \theta/d, (\theta + 1)/d, \cdots, \{\theta + (d-1)\}/d \}$ in the following
manner:
\[L_\theta(\eta):=\left\{\begin{array}{ll} m &
\mathbf{if} \hspace{2mm}  \eta\in\Bigl(\{\theta+(m-1)\}/d, (\theta+m)/d\Bigr)\\
\left( m_1, m_2\right)& \mathbf{if}
 \hspace{2mm}  \eta=\{\theta+(m_2-1)\}/d=(\theta+m_1)/d\\
\end{array}\right.\]
for some $m, m_1, m_2 \in \{ 0, 1,\cdots, d-1 \}$.

The infinite sequence $I_\theta(\eta):= \{ L_\theta(\eta), 
L_\theta(d\eta), L_\theta(d^2\eta), \ldots \}$ is called
the \emph{$\theta$-itinerary of~$\eta$} under the d-tupling map.  We call the special
itinerary $K(\theta):=I_\theta(\theta)
= \{ L_\theta(\theta), L_\theta(d\theta), L_\theta(d^2\theta),\ldots \}$ the
\emph{kneading sequence of~$\theta$}.
 
The symbols $\left( 0,  1 \right), \left( 1, 2 \right), \cdots,\left( d-2, d-1 \right), \left( d-1, 0 \right)$
are called \emph{boundary symbols}.  Sometimes we replace them by an
asterisk ($\ast$).
 
It is convenient to write~$K(\theta_1)=K(\theta_2)$ for
angles~$\theta_1,\theta_2$ if both angles have matching boundary symbols at
the same entries and all other symbols coincide.
\end{definition}

In the following theorem we define a partition of initial kneading
sequences and do most of the proof of the induction step of
Theorem~\ref{t:exactlytwo}.

\begin{theorem}[The Induction Step]
\label{t:parks}
Let~$n \geq 2$ be an integer. Suppose that the root of every hyperbolic
component with period~$n-1$ or lower is the landing point of exactly
two parameter rays. Then any two parameter rays with
angles~$\theta_1,\theta_2$ of exact ray period~$n$ can land at the
same parameter only if~$K(\theta_1)=K(\theta_2)$.
\end{theorem}

\begin{proof}
We claim that there is a partition~$P_{n-1}$ of~$\mathbb{C}$ such that every
parameter ray with exact ray period~$n$ together with its landing
point is completely contained in an open component of~$P_{n-1}$.
Moreover, for any two parameter rays~$\mathcal{R}_{\theta_1}, \mathcal{R}_{\theta_2}$
which are both in the same open component of~$P_{n-1}$ the kneading
sequences of~$\theta_1$ and~$\theta_2$ coincide in the first~$n-1$
entries. Then the theorem follows because any parameter ray at an angle
of exact period~$n$ has a kneading sequence of period~$n$ and
the~$n$-th entry of the kneading sequence is~($\ast$).

We construct such a partition: let~$\Theta_k$ be the set of all angles
with exact period~$k$ and~$\Lambda_k$ the set of multiplier maps of
the~$k$-periodic hyperbolic components. We define
$\displaystyle P_{n-1}:= \bigcup_{k=1}^{n-1}\left( \bigcup_{\theta \in \Theta_k} \mathcal{R}_{\theta}
\cup \bigcup_{\lambda_{\overline{H}}\in\Lambda_k} \lambda_{\overline{H}}^{-1}\left( \left[ 0, 1\right] \right)\right)$
and assert that~$P_{n-1}$ is a partition with the required properties.
By construction~$P_{n-1}$ (together with the components
of~$\mathbb{C} \setminus P_{n-1}$) is a partition of~$\mathbb{C}$.  Parameter rays with exact
ray period~$k$ land at a parameter which has a parabolic orbit with
exact ray period~$k$ (Lemma~\ref{per-rays-land}).  For a
hyperbolic component~$H$ the inverse
image $\lambda_{\overline{H}}^{-1}\left( \left[ 0, 1\right] \right)$ is the set of all internal rays
with angle~$0$. Each of these~$d-1$ internal rays lands at a root or
co-root of~$H$ and conversely the root and every co-root of~$H$ is a
landing point of one of these internal rays. It follows that every
parameter ray of period~$n$ together with its landing point is contained in one of the open components of~$\mathbb{C} \setminus P_{n-1}$.
 
Now assume that two parameter rays~$\mathcal{R}_{\theta_1}, \mathcal{R}_{\theta_2}$
are both contained in the same open component of~$\mathbb{C} \setminus P_{n-1}$. We know that
every hyperbolic component has~$d-2$ co-roots and exactly one root. Also by assumption, exactly two (resp. one) parameter ray(s) land at every root (resp. co-root) of hyperbolic components of period $k$, for $k \in \{ 1,\cdots,n-1 \}$. It follows that on the boundary of every hyperbolic component of period~$k$ ($k \in \{ 1,\cdots,n-1 \}$), exactly~$d$ parameter rays of period~$k$
land. Thus, for every~$k \in \{ 1,\cdots,n-1 \}$ the number of angles which are
in~$\Theta_k\cap \left( \theta_1,\theta_2 \right)$ is~$m\cdot d$ for
some~$m \in \mathbb{N}\cup \{ 0 \}$. This yields that the~$k$-th entry
($k \in \{ 1,\cdots,n-1 \}$) of~$K(\theta)$ is incremented~$m\cdot d$
times as~$\theta$ travels from~$\theta_1$ to~$\theta_2$. Therefore, $\theta_1$ and $\theta_2$ have the same kneading sequences.
\end{proof}

\begin{theorem}[Different Kneading Sequences]
\label{t:diffks}
Consider the angles of periodic rays landing at the characteristic point of the parabolic orbit in the dynamical plane of $c_0$, where $c_0$ is a root. Then all these angles have pairwise different kneading sequences, except for possibly the two characteristic angles.
\end{theorem}

\begin{proof}
We introduce notations: let~$z_1$ be the characteristic point
and~$\mathcal{R}_{\theta_1}^{c_0}, \cdots, \mathcal{R}_{\theta_s}^{c_0}$ the dynamical rays
landing at~$z_1$. We call the associated orbit portrait $\mathcal{P}$. For~$s = 2$ (where $s$ is the number of dynamical rays landing at $z_1$) there is nothing to prove, so we
assume~$s \geq 3$ and are automatically in the non-primitive case
(Lemma~\ref{folklore}). Denote the exact period of the
angles~$\theta_1,\cdots,\theta_s$ by~$n$ and the orbit period of the
parabolic orbit by~$k = n/s$. Without restriction we
assume~$k \geq 2$. For~$i \in \{ 1,\cdots,s \}$ denote the~$d$ inverse
images of~$\theta_i$ with respect to the~$d$-tupling map
by~$\theta_i^{(l)}:= \left(\theta_i+l \right)/d \in \mathbb{S}^1$ and the landing point
of~$\theta_i^{(l)}$ by~$z^{(l)}_0$,
$l \in \{ 0,\cdots, d-1 \}$. The points ~$z^{(0)}_0,\ldots,z^{(d-1)}_0$ do not
depend on the choice of a specific angle~$\theta_i$
($i \in \{ 1,\cdots,s \}$); in fact they are just the pre-images of $z^1$ under $f_{c_0}$. Finally, let $H$ be the hyperbolic component of period $n$ with $c_0 \in \partial H$ and $c_1$ a center of $H$.

By Corollary~\ref{c:phe}
and Lemma~\ref{l:brycf} there are continuous
functions~$z^{(0)},\cdots,z^{(d-1)}$ on~$H \cup \{ c_0 \}$ such
that~$z^{(l)}(c_0)=z^{(l)}_0$ and at~$z^{(l)}(c)$ land the
dynamical rays at the same angles as at~$z^{(l)}_0$ for all~$i \in \{ 0,1,\cdots,d-1 \}$
and~$c \in H \cup \{ c_0 \}$. The points~$z^{(l)}(c_1)$ lie on the
boundary of the critical Fatou component~$U_0$
of~$f_{c_1}$. Let~$\Gamma$ be the Hubbard tree of~$c_1$, $U_1$ the
characteristic Fatou component and~$\gamma$ be the regular arc which
connects the unique intersection point~$\Gamma \cap \partial U_1$ with
the critical value~$c_1$. Then~$\gamma$ has~$d$ inverse images and each of
them lies in~$\overline{U_0}$ and connects the critical point with one
of the points~$z^{(l)}(c_1)$. Therefore,~$\mathbb{C} \setminus P_{\theta_i}$ 
(where, $P_{\theta_i}:= f_{c_1}^{-1}\left( \gamma \right)\cup\bigcup_{l = 0}^{d-1} \mathcal{R}_{\theta_i^{(l)}}^{c_1}$)
has precisely~$d$ open components for each $i \in \{ 1, 2, \cdots, s \}$. We label the
boundary~$P_{\theta_i}$ by~($\ast$) and the component containing the
critical value~$c_1$ by~$L_1$. The subsequent components are labelled (as $L_i$) in anti-clockwise direction. By construction the branch of~$\Gamma \setminus \overline{U_0}$ on which the critical value lies is contained in the component with
label~$L_1$. Since~$f_{c_1}$ is orientation preserving this implies that
the label of any branch of~$\Gamma \setminus \overline{U_0}$ does not depend
on~$P_{\theta_i}$. Now it's easy to see that for a fixed $j \neq k-1 (mod \hspace{1mm} k)$, all the dynamical rays at angles $\{ d^j \cdot \theta_i : i = 1,2,\cdots,s \}$ have the same label with respect to the corresponding partition $P_{\theta_i}$. Thus the kneading sequences of the ~$\theta_i$'s can differ only in the $(mk-1)$-th position, for some $m \in \mathbb{N}$.

\begin{figure}[!ht]\label{kneading}
\begin{center}
\includegraphics[scale=0.35]{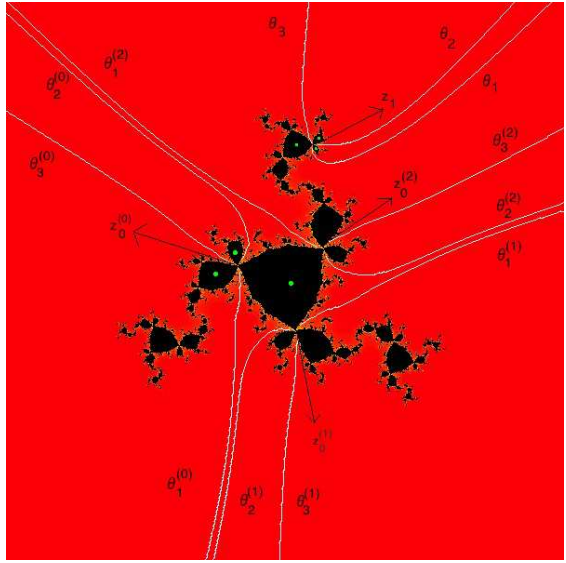}
\end{center}
\caption{The Julia set of a polynomial $z \mapsto z^3 + c, (c \approx 0.2253 + 0.9414 i)$ with a 6-periodic super-attracting orbit (marked in green). With the notation of the proof, the dynamical rays landing at $z_1$ have angles $\theta_1 = 92/728, \theta_2 = 100/728$ and $\theta_3 = 172/728.$ The rays $\theta_i^{(l)}$ landing at the three pre-images of $z_1$ are also drawn in.}
\end{figure}

The kneading sequences of all~$\theta_i$ except for two angles are
pairwise different at an~$(mk-1)$-th position: the~$(mk-1)$-th entry
of the kneading sequence of~$\theta_i$ is just the label
of the ray $\mathcal{R}_{d^{(mk-1)} \cdot \theta_i}^{c_1}$ with respect
to the partition $P_{\theta_i}$. Let the dynamical rays landing at the parabolic point on the boundary of the critical Fatou component be $ \mathbb{S} = \{ \mathcal{R}_{\theta_1^{(l)}}^{c_1}, \mathcal{R}_{\theta_2^{(l)}}^{c_1}, \cdots, \mathcal{R}_{\theta_s^{(l)}}^{c_1}\}$, for some $ l \in \{ 0, 1, \cdots, d-1 \}$. Now, $K(\theta_i)=K(\theta_j)$ only if the number of rays (amongst $\mathbb{S}$) which lie in a given component with label $L_r$ with respect to the partition $P_{\theta_i}$ is equal to the number of such rays that lie in the corresponding component with label $L_r$ with respect to the partition $P_{\theta_j}$. However, if at least two rays at angles in $\mathbb{S}$ have different labels with respect to~$P_{\theta_i}$, then the number of rays with the smaller label is different with respect to~$P_{\theta_i}$ and~$P_{\theta_j}$ for~$i \neq j$. Therefore, all these dynamical rays must have the same label. This is only possible if none of the rays at angles in $\mathbb{S}$ lies in the connected component of $\mathbb{C} \setminus (\mathcal{R}_{\theta_i^{(l)}}^{c_1} \cup \mathcal{R}_{\theta_j^{(l)}}^{c_1}\cup
\{z^{(l)}(c_1)\})$ containing ~$U_0$, i.e.~if~$\theta_i$ and~$\theta_j$ are the characteristic angles.
\end{proof}

\noindent We finish the proof of the Structure Theorem in the periodic
case:

\begin{theorem}[Precisely Two Parameter Rays Land at Every Root ]
\label{t:exactlytwo}
Every root~$c_0$ is the landing point of exactly two parameter
rays. The angles of the parameter rays landing at~$c_0$ are
the characteristic angles of the parabolic orbit of~$c_0$.
\end{theorem}

\begin{proof}[Proof by induction on the ray period~$n$ of~$c_0$]
For~$n=1$, there are $d-1$ co-roots; but there is a unique co-root which is the landing point of the parameter rays at
angles~$0$ and~$1$ that we consider as two rays in this case. Assume
that the roots of all hyperbolic components with period~$n-1$ or lower
are the landing points of exactly two parameter rays. We obtain by
Theorem~\ref{t:parks} that only the parameter rays at~$n$-periodic angles
with same kneading sequences can land at the root~$c_0$ of any hyperbolic
component with period~$n$. Note that a parameter ray
with a given angle can land at~$c_0$ only if the dynamical ray at the
same angle lands at the characteristic point~$z_0$ of the parabolic
orbit (Theorem~\ref{t:perchar}) and that the angles of all the dynamical rays
landing at~$z_0$, except possibly for the characteristic
angles~$t^{-}$ and $t^{+}$, have different kneading
sequences. Therefore, the only candidates for landing at $c_0$ are the two parameter rays at
angles~$t^{-},t^{+}$. By Corollary~\ref{At least two rays}, we know that they indeed land at $c_0$. This finishes
the induction.
\end{proof}

\section{Pre-periodic Parameter Rays}\label{s:preper}
In this section, we record the landing properties of the parameter rays of the multibrot sets at pre-periodic angles. The generalization of the following results from the quadratic case is straight-forward in this case and does not require any new technique. We refer the readers to \cite[Section~4]{S1a} for a more comprehensive account on the combinatorics of the parameter rays (of the Mandelbrot set) at pre-periodic angles.

\begin{definition}[Misiurewicz Point]
A parameter~$c$ for which the critical orbit is strictly pre-periodic is called a \emph{Misiurewicz point}.
\end{definition}

\begin{theorem}[Pre-periodic Parameter Rays Land]
\label{t:prep-land}
Every parameter ray at a pre-periodic angle~$\theta$ lands at a
Misiurewicz point~$c_0$. The corresponding dynamical ray~$\mathcal{R}_\theta^{c_0}$ lands at
the critical value~$c_0$.
\end{theorem}

\begin{proof}
See the proof of the pre-periodic case in \cite[Theorem~1.1]{S1a}.
\end{proof}

\begin{theorem}[Every Misiurewicz Point is a Landing Point]
\label{t:everymislp}%
Every Misiurewicz parameter is the landing point of a parameter ray at a pre-periodic angle.
\end{theorem}

\begin{proof}
See the proof of the pre-periodic case in \cite[Theorem~1.1]{S1a}.
\end{proof}

\begin{theorem}[Number of Rays at a Misiurewicz Point]
\label{t:noray}
Suppose that a pre-periodic angle $\theta$ has pre-period $l$ and period $n$. Then the kneading sequence $K \left( \theta \right)$ has the same pre-period $l$, and its period $k$ divides $n$. If $n/k > 1$, then the total number of parameter rays at pre-periodic angles landing at the same point as the ray at angle $\theta$ is $n/k$; if $n/k = 1$, then the number of parameter rays is $1$ or $2$.
\end{theorem}

\begin{proof}
See \cite[Lemma 4.4]{S1a}.
\end{proof}

\bibliographystyle{alpha}
\bibliography{multibrot}

\end{document}